\documentclass[11pt,leqno]{amsart}

\usepackage{figsize,subfigure,graphicx,ifthen,calc,psfrag}
\usepackage{enumerate,amsthm,amsmath,amssymb,amsxtra,comment,graphicx,psfrag}
\usepackage{xspace}

\usepackage{hyperref}
\usepackage{harvard}
\usepackage{nameref}
\usepackage{hhline}
\usepackage{verbatim}
\numberwithin{equation}{section}
\renewcommand{\baselinestretch}{1.1}

\begin{document}

%%%%%%%%%%%%%% ***** definitions  *************

\font\srm=cmr8 at 8. truept
\def\bp{{\bar \partial}}
\def\pd{{\partial_\tau}}
\def\bU{{\bar U}}
\def\bW{{\bar W}}
\def\bdel{{\bar{\delta}}}
\def\ubdel{{\underline{\delta}}}
\def\ubk{{\underline k}}
\def\Un2{{U^{n+\frac 12}}}
\def\n2{{n+\frac 12}}
\def\IL{{{I_{\hbox{\srm L}}}}}
\def\Th{{\Theta}}
\def\T_h{{{\mathcal T}_h}}
\def\La{{\Lambda}}
\def\F{{\mathcal{F}}}
\def\G{{\mathcal{G}}}
\def\E{{\mathcal{E}}}
\def\H{{\mathcal{H}}}
\def\<{{\langle }}
\def\>{{\rangle }}
\def\V{{\mathcal{V}}}
\def\U{{\mathcal{U}}}
\def\sg{{\sigma_\gamma}}

\def\tends{\rightarrow}

\def\al{\alpha}
\def\f{\varphi}
\def\wU{\widehat U}
\def\wA{\widehat A}
\def\wa{\widehat a}
\def\he{\hat e}
\def\hU{\hat U}
\def\hR{\hat R}
\def\gam{\gamma}
\def\be{\beta}
\def\la{\lambda}
\def\th{\vartheta}
\def\t{\tau}
\def\eps{\varepsilon}
\def\ds{\displaystyle}
\def\C{{\mathbb C}}
\def\Re {{\mathbb R}}
\def\P{{\mathbb P}}
\def\N{{\mathbb N}}
\def\R+{{\Bbb R}_*^+}
\def\Z{{\mathbb Z}}
\def\e{{\rm e}}
\def\i{{\rm i}}
\def\tr|{|\!|\!|}

\def\tm{\tilde{m}}
\def\tp{\tilde{p}}

\def\tf{f}
\newcommand{\ex}{\mbox{E}}
\newcommand{\var}{\mbox{Var}}
\newcommand{\cov}{\mbox{Cov}}
\newcommand{\cum}{\mbox{cum}}
\newcommand{\beps}{\mbox{\boldmath  $\varepsilon $}}

\newtheorem{lemma}{Lemma}[section]
\newtheorem{theorem}{Theorem}[section]
\newtheorem{remark}{Remark}[section]
\newtheorem{corollary}{Corollary}[section]
\newtheorem{notation}[lemma]{Notation}
\newtheorem{definition}[lemma]{Definition}
\newtheorem{subsec}[lemma]{}
%\vspace*{2.0cm}
%\renewcommand{\aa}{{\mbox{\boldmath$a$}}}

\newcommand{\Proof}{{\sc Proof.}\ }

\def\Box{\vbox{\offinterlineskip\hrule

        \hbox{\vrule\phantom{o}\vrule}\hrule}}

\newcommand{\bull}{\rule{2mm}{2mm}}

%***********************************************************

%\documentstyle[12pt,leqno]{amsart}
%\input epsf
\textheight22.25cm
\textwidth17.5cm
\begin{comment}
\oddsidemargin 0cm \evensidemargin 0cm \topmargin 0pt \headheight
0pt \headsep 0pt \footskip 1cm \topskip 0cm \textheight 23cm
\textwidth 16.2cm \frenchspacing \font \fivesans              =
cmss10 at 5pt \font \sevensans             = cmss10 at 7pt \font
\tensans               = cmss10
\end{comment}
\newfam\sansfam
%\textfont\sansfam=\tensans\scriptfont\sansfam=\sevensans
%\scriptscriptfont\sansfam=\fivesans
\def\sans{\fam\sansfam\tensans}
% Diese Zeile fuer Absentmindedness aforementioned !
\parskip 6pt plus 1pt \parindent 0cm \partopsep 3pt plus 2pt minus 2pt
\def\IC{{\mathbb C }}
\def\IR{{\mathbb R}}
\def\IP{{\mathbb P}}
\def\IN{{\mathbb N}}
\def\IZ{{\mathbb Z}}
\def\IB{{\rm I\kern-0.19em B }}
\def\KIN{{\rm I\kern-0.12em N }}
\def\IQ{{\rm {\vrule depth-0.1ex height 1.5ex}\kern-0.30em Q }}
\newcommand{\vecz}[2]{\left( \hspace{-1mm}\begin{array}{c}#1\\#2
                          \end{array} \hspace{-1mm}\right)}
\newcommand{\vecd}[3]{\left(\hspace{-1mm} \begin{array}{c}#1\\#2\\#3
                          \end{array} \hspace{-1mm}\right)}
\newcommand{\bgamma}{\mbox{\boldmath  $\gamma $}}
\newcommand{\balpha}{\mbox{\boldmath  $\alpha $}}
%**********************************************************************
%
%  Sonderzeichen fuer spezielle Mengen
%
%**********************************************************************
\global\font\fontZsetten=cmss12 \global\font\sevenrm=cmr7
\def\fontZsetZD{\hbox{\fontZsetten Z}}
\def\fontZsetZS{{\sevenrm Z}}
\def\ZsetD{\fontZsetZD\mkern-7.2mu\fontZsetZD}
\def\ZsetS{\fontZsetZS\mkern-9.2mu\fontZsetZS}
\def\NsetD{I\mkern-6mu N}
\def\RsetD{I\mkern-6mu R}
\def\QsetD{Q\mkern-9.9mu\raise.02ex\hbox{\vrule height1.3ex width.12ex%
depth0ex}\mkern11mu}
\def\QsetS{Q\mkern-9.9mu\raise.02ex\hbox{\vrule height.9ex width.1ex%
depth0ex}\mkern11mu}
\def\CsetD{C\mkern-9.9mu\raise.02ex\hbox{\vrule height1.3ex width.12ex%
depth0ex}\mkern11mu}
\def\CsetS{C\mkern-9.9mu\raise.02ex\hbox{\vrule height.9ex width.1ex%
depth0ex}\mkern11mu}
\def\Zset{{\mathbb Z}}
\def\Nset{\mathpalette{}\NsetD}
\def\Rset{\mathpalette{}\RsetD}
\def\Qset{\mathchoice{\QsetD}{\QsetD}{\scriptstyle\QsetS}{\scriptstyle\QsetS}}
\def\Cset{\mathchoice{\CsetD}{\CsetD}{\scriptstyle\CsetS}{\scriptstyle\CsetS}}
\newcommand{\D}{\displaystyle}
\newcommand{\Sc}{\scriptstyle}
%\newcommand{\E}[1]{\mbox{E}\,\left\{#1\right\}}
%\newcommand{\V}[1]{\mbox{Var}\,\left\{#1\right\}}
%\newcommand{\cov}[2]{\mbox{cov}\,\left\{#1,#2\right\}}
%\newcommand{\bm}[1]{\mbox{\boldmath$#1$}}
%\newcommand{\bigtimes}[2]{\renewcommand{\arraystretch}{0.4}
%         \begin{array}{c} {\Sc #2} \\ \mbox{\Large $\times$} \\ {\Sc #1}
%         \end{array} \renewcommand{\arraystretch}{1} \!\!}
%\newcommand{\otimesl}[2]{\renewcommand{\arraystretch}{0.5}
%       \begin{array}{c} {\Sc #2} \\ {\D \otimes} \\ {\Sc #1} \end{array}
%       \renewcommand{\arraystretch}{1} }
%************************************************
% grosse griechische Buchstaben
%************************************************
\def\Alpha{{\rm A}}
\def\Beta{{\rm B}}
\def\Epsilon{{\rm E}}
\def\Zeta{{\rm Z}}
\def\Eta{{\rm H}}
\def\Iota{{\rm I}}
\def\Kappa{{\rm K}}
\def\Mu{{\rm M}}
\def\Nu{{\rm N}}
\def\omikron{{\sl o}}
\def\Omikron{{\rm O}}
\def\Rho{{\rm P}}
\def\Tau{{\rm T}}
\def\Chi{{\rm X}}
\newlength{\labst}

%\newtheorem{Definition}{Definition}[chapter]
%\newtheorem{Theorem}{Theorem}[chapter]
%\newtheorem{Lemma}{Lemma}[chapter]
%\newtheorem{Corollary}{Corollary}[chapter]
%\newtheorem{Proposition}{Proposition}[chapter]
%\setlength{\labst}{2ex}
%\newcommand{\abst}{\vspace{\labst}}
%\newenvironment{Proof}{{\bf Proof:}}{\hfill\rule{2mm}{2mm}\abst}

%\renewcommand{\baselinestretch}{1.02}
%\begin{document}
\title[\sf Bootstrapping   Whittle Estimators]{Bootstrapping  Whittle Estimators}
\author[\sf J.-P. Kreiss]{Jens-Peter Kreiss}
\address{Technische Universit\"at Braunschweig, Institut f\"ur Mathematische Stochastik,
                 Universit\"atsplatz 2, D--38106 Braunschweig, Germany.}
\author[\sf E. Paparoditis]{Efstathios Paparoditis}
\address{University of Cyprus,
         Department of Mathematics and Statistics,
         1678 Nicosia,
         Cyprus.}
\date{\today}


\begin{abstract}
Fitting parametric models  by optimizing frequency domain objective functions  is
an attractive approach of parameter estimation in  time series analysis. 
Whittle estimators are a prominent example in this context.
Under weak conditions   
and  the  (realistic)  assumption   that the true spectral density of the underlying process does not necessarily belong to the parametric class of spectral densities fitted, 
the distribution of Whittle estimators typically depends on difficult to estimate characteristics 
of the underlying process.  This makes  the implementation of asymptotic results  for  the construction of  confidence intervals or for  assessing  the variability of   estimators, difficult in practice.  This paper  proposes a frequency domain bootstrap method  to estimate the distribution of Whittle estimators which is asymptotically 
valid 
under   assumptions that  not only allow for (possible) model misspecification but also for  weak dependence conditions   which  are satisfied by  a wide range of stationary stochastic processes. Adaptions of the bootstrap procedure  developed to incorporate   different modifications  of Whittle estimators  proposed in the literature, like for instance, tapered,  de-biased or boundary extended Whittle estimators,  are also  considered. 
Simulations    demonstrate the capabilities of the bootstrap method proposed  and its good finite sample performance.  A real-life data analysis also is presented.   
% which are very g of the underlying process and on condition on the parametric spectral density 
%hese refer not only to the general conditions on the underlying stochastic process and under the 

\end{abstract}

\maketitle

%=================================================================

\section{Introduction}
\label{se.Intro}

%Fitting parametric models to a stretch of observed time series   by 
Optimizing frequency domain objective functions  is
an attractive approach of fitting  parametric models to observed time series.  Let $ X_1, X_2, \ldots, X_n$ be  a time series stemming from a stationary process  $\{X_t,t\in\Z\}$  and assume that 
this process  possesses  a spectral density $f$. 
Let 
%\begin{equation} \label{param-f}
%{\mathcal F}_\theta = \{ f_\theta, \theta \in \Theta \subseteq {\mathbb R}^m, \inf _{\theta\in\Theta}\inf_{\lambda\in [-\pi,\pi]}f_\theta(\lambda) > \delta >0\},
%\end{equation}
$ {\mathcal F}_\theta$ be a  family of parametric spectral densities, where $ f_\theta \in {\mathcal F}_\theta$ is  determined by a $m$-dimensional parameter vector $ \theta$ and suppose that we are interested in fitting to $ X_1, X_2, \ldots, X_n$ a model from the class ${\mathcal F}_\theta$.
%The spectral densities  $ f_\theta$  belonging to the family ${\mathcal F}_\theta$  
%are   bounded from bellow  away from zero for all frequencies $\lambda \in [-\pi,\pi]$ and for all $ \theta \in \Theta$. Model identifiability requires  that  $ \theta_1\neq \theta_2$ implies   $f_{\theta_1}(\lambda) \neq f_{\theta_2}(\lambda)$  on a set of frequencies $\lambda$ with positive Lebesgue measure.  
Notice that we do not  assume  $ f\in {\mathcal F}_\theta$, that is,  we allow for  the practical important case of  model misspecification where 
  the parametric model class considered  does   not necessarily  contain  the true spectral density  $f$.  
  
 Several approaches for  selecting  a frequency domain objective function and consequently for developing a  frequency domain procedure to fit  parametric models  exist; we refer here to Taniguchi (1987) and to Dahlhaus and Wefelmeyer  (1996) for  examples.  In this context, Whittle estimators play an important role. This is due to the fact that Whittle  estimators are  computationally fast and they are obtained via minimizing a frequency domain approximation of  the (Gaussian) log likelihood function.  Moreover,  for a variety of models and under different assumptions,   Whittle estimators are 
   asymptotically normal, asymptotically equivalent to the exact maximum likelihood estimators and asymptotically efficient (in Fisher-sense) if $ f\in {\mathcal F}_\theta$; see  Sykulski et al. (2019) and   Subba Rao and Yang (2020) for a recent discussion of the related literature. Furthermore and even if $ f \notin {\mathcal F}_\theta$, Whittle estimators   retain certain efficiency properties; Dahlhaus and Wefelmeyer (1996). Despite these nice properties,  however, the limiting distribution of Whittle's estimators  is affected by  characteristics of the process which make the  implementation of  asymptotic results  for assessing their  variability   or for constructing   confidence intervals,  difficult in practice. To elaborate,  allowing for possible model misspecification and avoiding restrictive structural assumptions for the underlying process class, like for instance linearity assumptions,   the limiting distribution of Whittle estimators typically depends  
on   the parametric spectral density from the class $ {\mathcal F}_\theta$ which ``best fits'' the data, say $ f_{\theta_0}$, the true spectral  density $f$ as well as  the entire 
   fourth order cumulant structure    of the underlying process $ \{X_t; t \in \Z\}$.  Estimation of the last mentioned quantity is  a rather difficult problem.

In situations like the above, bootstrapping may offer an alternative to classical large sample approximations. 
Bootstrapping Whittle estimators has been discussed  in the literature  under  certain  structural assumptions on the underlying process  $\{X_t,t\in \Z\}$. It is  typically assumed 
 that  $\{X_t,t\in\Z\}$   is a linear process driven by i.i.d. innovations  and   that Kolmogorov's formula holds true; see Dahlhaus and Janas (1996) and Kim and Nordman (2013).
Such   structural assumptions lead to a simplification of the limiting distribution of Whittle estimators and consequently of   the  features of the underlying linear process that the bootstrap procedure has  to appropriately mimic in order to be consistent. To elaborate,  notice first
that, under  linearity assumptions, the  problem of estimating the variance $\sigma^2$ of the i.i.d. innovations driving   the linear process can be separated from
the problem of estimating the remaining  coefficients of the parameter vector,   denoted by  $ \tau \in \Re^{m-1}$, i.e.,  $\theta=(\sigma^2, \tau)$. Second and more importantly, 
under the linearity assumption,    the limiting distribution of  the Whittle estimator of the parameter  part  $ \tau$,   does not depend on the fourth order cumulants  of the  process; see Section 2 for  details.  
Therefore,  if one is  solely   interested   in estimating  the distribution of the Whittle estimator of  the innovation free part $\tau$, then  standard frequency domain   bootstrap procedures  which 
generate   pseudo periodogram ordinates  that  are independent across frequencies,  can   successfully be applied.
In this context, the  multiplicative bootstrap, see  Hurvich an Zeger (1987), Franke and H\"ardle (1992)
and Dahlhaus and Janas (1996), or the local periodogram bootstrap, Paparoditis and Politis (1999), are  consistent.   
However,  the pseudo  periodogram ordinates generated by the aforementioned bootstrap procedures   are independent across frequencies. Therefore,   these  
bootstrap procedures  are  not able to imitate the covariance structure of the periodogram ordinates which is responsible for the fact that the fourth structure of the process shows up in the limiting distribution of Whittle estimators.  As a  consequence, these procedures fail  in all cases where  
 the  fourth order cumulants  of the  process affect the distribution of  interest. 

In this paper we present a  frequency domain, hybrid  bootstrap procedure for Whittle  estimators which is valid  under  weak assumptions on the underlying process $ \{X_t,t\in\Z\}$ and which also covers the practical important case  where the true spectral density $f$ does not necessarily belong to the parametric class  $ {\mathcal F}_\theta$.
The procedure consists of two main parts: A multiplicative frequency domain bootstrap part and a part based on the convolution of resampled periodograms of subsamples. The latter is an adaption to the frequency domain of  the convolved subsampling idea proposed in Tewes et al. (2019).   The two parts contribute differently  and complementary in  estimating the distribution of interest. The multiplicative part of the bootstrap procedure  is  used to estimate all features   of the distribution of Whittle estimators   including the  parts of the limiting covariance matrix that depend  on the second order structure of the underlying process and of the parametric model fitted to the time series at hand.  However, the components  of the covariance matrix of these  estimators that depend on the fourth order structure of the underlying process $ \{X_t,t\in\Z\}$,   are estimated using the part of the bootstrap procedure which is based on the  convolution of resampled periodograms of  subsamples.  Putting the two parts together in an appropriate way, leads to a bootstrap procedure which is asymptotically valid under  conditions on the dependence structure of the process $\{X_t,t\in\Z\}$ which go far beyond linearity and at the same time   appropriately captures the effects of (possible) model misspecification on the distribution of interest. 

The frequency domain, hybrid approach proposed  in this paper and  which uses convolution of resampled periodograms of subsamples together with the multiplicative periodogram bootstrap, is related to the proposal  of  Meyer et al. (2020). However and additional to differences in the  technical tools used  to establish bootstrap consistency, the merging of the multiplicatve and of the convolved part of the bootstrap procedure presented here,  is  different, more involved and tailormade for Whittle estimators.  Furthermore, we   show, how the described bootstrap procedure  can appropriately be modified, respectively, extended to incorporate several modifications of Whittle estimators which have been proposed  in order to improve the finite (small) sample bias of these estimators. This  concerns  tapered Whittle estimators,  Dahlhaus (1988), so-called de-biased  Whittle estimators, Sykulski et al. (2019),  and improvements of Whittle's  quasi Gaussian  likelihood approximation based on boundary corrected periodograms; see Subba Rao and Yang (2020).  The corresponding  extensions of the frequency domain bootstrap proposed in this context  are  novel and  of 
%independent 
 interest  on their own.

The paper is organized as follows. In Section 2  we review some basic results on Whittle estimators which are important for our subsequent development  of   the bootstrap. The basic frequency domain bootstrap procedure introduced is presented in Section 3. Section 4 is devoted to the derivation of theoretical results  and  establishes  consistency of the  bootstrap.  Section 5 deals with the extensions  of the basic bootstrap procedure in order to  incorporate the aforementioned  modifications of   standard Whittle estimators.  Section 6 discusses  some issues related to the practical implementation of the bootstrap algorithm,  presents  simulations which investigate the finite sample performance of the new bootstrap method and make comparisons  with  the asymptotic Gaussian approximation, respectively, the multiplicative periodogram bootstrap. A real-life data application also is discussed.
 Auxiliary  lemmas as well as proofs of the main results are deferred  to 
Section 7.

\section{Whittle Estimators}
\label{se.Whittle}

In his PhD thesis Whittle introduced a frequency domain approximation of the log-likelihood function 
%$ l_n(\cdot)$
 of a stationary Gaussian time series (cf. Whittle (1951) and  Whittle (1953)). This approximation can be written as 
 \begin{align*}
 l_n(\theta) & %=  2n\log (2\pi) + \sum_{j\in {\mathcal F}_n}\Big\{ \log f(\lambda_{j,n}) + \frac{\displaystyle |J_n(\lambda_{j,n})|^2 }{\displaystyle 2\pi f(\lambda_{j,n})} \Big\} \\
  %&
  = 2n\log (2\pi) + \sum_{j\in {\mathcal F}_n}\Big\{\log f(\lambda_{j,n}) + \frac{\displaystyle I_n(\lambda_{j,n}) }{\displaystyle f(\lambda_{j,n}) } \Big\},
 \end{align*}
where 
%$  |J_n(\lambda_{j,n})|^2/(2\pi) = 
$ I_n(\lambda_{j,n})$, denotes  the periodogram of  the time series $ X_1, X_2, \ldots, X_n$ evaluated at the Fourier frequency $\lambda_{j,n}= 2\pi j/n \in {\mathcal F}_n$ and  $ {\mathcal F}_n=\{-[(n-1)/2], \ldots, [n/2]\}$ is the set of Fourier frequencies. 
Ignoring the first additive term  and approximating the integral over the set  $ {\mathcal G}(n)=\{ -N, -N+1, \ldots, -1, 1, \ldots, N\}$, where $ N=[n/2]$, 
   Whittle's  approximation to the log-likelihood function used in this paper, is given by 
\begin{equation} \label{eq.whittle-lk-disc}
D_n(\theta,I_n) = \frac{1}{n}\sum_{j\in {\mathcal G}(n)}\Big\{\log f_\theta (\lambda_{j,n}) + \frac{\displaystyle I_n(\lambda_{j,n})}{\displaystyle f_\theta(\lambda_{j,n})}\Big\}.
\end{equation} 
%which is commonly used in applications. 
%If we now think of fitting a parametric model   from some parametric  class $ {\mathcal F}_\theta$ to the underlying stationary time series,  %then the spectral density $ f_\theta$ depends on some parameter $ \theta \in \Theta$.  
A minimizer $\widehat{\theta}_n$ of 
%\[ \theta  \mapsto  2n\log (2\pi) + \sum_{j\in {\mathcal F}_n}\Big\{\log f_\theta(\lambda_{j,n}) + \frac{\displaystyle I_n(\lambda_{j,n}) }{\displaystyle f_\theta(\lambda_{j,n}) }\Big\},\]
\begin{equation} \label{eq.Whittle-est}
\theta  \mapsto   D_n(\theta, I_n),
\end{equation}
is  called a Whittle estimator of $ \theta$. Note  that   (\ref{eq.whittle-lk-disc})  can be considered as  a Riemann sum approximation of 
%A Whittle estimator  of $ \theta$ is obtained   by minimizing the  so called Whittle likelihood function, that is the objective function,  
\begin{equation} \label{eq.whittle-lk}
D(\theta, I_n) =  \frac{1}{2\pi}\int_{-\pi}^\pi \Big\{\log f_\theta (\lambda) + \frac{\displaystyle I_n(\lambda)}{\displaystyle f_\theta(\lambda)}\Big\}d\lambda.
\end{equation}   
%where in applications the  discretized version,
% Let $ \widehat{\theta}_n$ be the estimator defined as  
%\begin{equation} \label{eq.Whittle-est}
%\widehat{\theta}_n = \arg \min_{\theta\in\Theta} D(\theta,I_n)
%\end{equation}  
Let 
\begin{equation} \label{eq.kulb1}
   D(\theta,f)=  \frac{1}{2\pi}\int_{-\pi}^\pi \Big\{\log f_\theta (\lambda) + \frac{\displaystyle f(\lambda)}{\displaystyle f_\theta(\lambda)}\Big\}d\lambda
\end{equation}   
and assume that  
%and suppose that 
 \begin{equation} \label{eq.theta0}
   \theta_0= \arg \min_{\theta\in\Theta} D(\theta,f),
  \end{equation} 
  exists and is unique. 
$ \widehat{\theta}_n$   is then   an estimator of $ \theta_0$ and 
$f_{\theta_0} $ denotes   the spectral density from the parametric family $ {\mathcal F}_\theta$ which best fits 
the spectral density $f$ of the underlying process $ \{X_t, t \in \Z\}$ in the sense of  minimizing  
 the divergence measure (\ref{eq.kulb1}).  Note  that    if $ f \notin {\mathcal F}_\theta$, then  the best approximating parametric spectral density  $f_{\theta_0}$  from the class $ {\mathcal F}_\theta$,  clearly depends on the particular divergence measure  
$D(\theta,f)$  associated with     Whittle's approximation of the log-likelihood function.  Observe  that  $\theta_0$ which minimizes    $D(\theta,f)$   is the same as the one  which 
 minimizes the so-called Kullback-Leibler information divergence. The later is given  for  Gaussian processes    by 
\[  \frac{1}{2\pi} \int_{-\pi}^\pi \Big\{\log \frac{f_\theta (\lambda)}{f(\lambda)}  + \frac{\displaystyle f(\lambda)}{\displaystyle f_\theta(\lambda)} -1\Big\}d\lambda;\]
see Dahlhaus and Wefelmeyer (1996). 
%This provides an additional  interpretation of Whittle estimators. 
%In the following we write for simplicity $ \widehat{\theta}$ for $\widehat{\theta}_n^W$.
%Instead of (\ref{eq.kulb1}), other divergence criteria   also can be  used in order to define frequency domain parameter estimators. We refer here to   Taniguchi (1987) 
% and to Dahlhaus and Wefelmeyer (1996) for different proposals. However,  we  restrict our considerations   to  Whittle estimators because of their prominent role in time series analysis.
Different modifications of  Whittle's likelihood  approximation (\ref{eq.whittle-lk})  have been proposed in the literature in order    to improve the finite sample behavior  
and more specifically the   bias of  the estimator $ \widehat{\theta}_n$. We mention here    the tapered Whittle likelihood proposed by Dahlhaus (1988), the debiased Whittle likelihood proposed by    Sykulski et al. (2019),   the boundary corrected  and the  hybrid Whittle likelihood approximation proposed by  Subba Rao and Yang (2020). For the sake of  a better presentation,  however, 
we will first focus on   the  standard    Whittle likelihood approximation  (\ref{eq.whittle-lk-disc}), respectively,  (\ref{eq.whittle-lk}).  Later  on, we will  elaborate on   how to appropriately modify the  basic bootstrap procedure  proposed in order 
 to take into account  the    aforementioned modifications/extensions  of Whittle estimators.

%It has been shown in the literature for a variety of models and assumptions that  $ \widehat{\theta}_n$ is asymptotically normal, asymptotically equivalent to the exact maximum likelihood and asymptotically efficient (in Fisher-sense); see Sukulski et al. (2019), Subba Rao and Yang (2020) for a discussion of the related literature.
%In the following we  briefly elaborate on the basic steps used to  derive the 
% (limiting) distribution of  Whittle's estimator and which are  important for our subsequent discussion of the bootstrap. 
% 

Let us briefly review the main ideas  involved in deriving the limiting distribution of the  estimator $ \widehat{\theta}_n$ and which are important for our subsequent 
discussion of the bootstrap.  Consider (\ref{eq.whittle-lk}) and (\ref{eq.kulb1} ), assume that $ f_\theta$  is sufficiently smooth with respect to $ \theta$ and  recall that 
$\widehat{\theta}_n$  and $ \theta_0$ satisfy the score equations
\begin{equation} \label{eq.tayl-1}
 \frac{\displaystyle \partial}{\displaystyle \partial \theta}D_n(\theta, I_n)\Big|_{\theta=\widehat{\theta}_n} = 0 \ \  \ \mbox{and} \  \ 
\frac{\displaystyle \partial}{\displaystyle \partial \theta}D(\theta, f)\Big|_{\theta=\theta_0} = 0,
\end{equation}
respectively. 
Using  a linear approximation of $  (\partial/ \partial \theta)D_n(\theta, I_n)\big|_{\theta=\widehat{\theta}_n}$ around $ \theta_0$ we  get,   taking into account (\ref{eq.tayl-1}), that   
\[ 0 =  \frac{\displaystyle \partial}{\displaystyle \partial \theta}D_n(\theta_0, I_n)  +  
 \frac{\displaystyle \partial^2}{\displaystyle \partial \theta\partial \theta^\top}D_n(\theta_0, I_n) (\widehat{\theta}_n - \theta_0) +  R_n.\]
 Notice that  for simplicity,   the notation $(\partial/\partial \theta)D_n(\theta_0, I_n) $ for $ 
   (\partial/\partial \theta)D_n(\theta, I_n)\big|_{\theta=\theta_0}$ has been used with an  analogue notation   for the matrix of second order partial derivatives 
   $ \partial^2 \big/\partial \theta\partial \theta^\top D_n(\theta_0, I_n)$. Provided that  
   the remainder $R_n$ is $ o_P(n^{-1/2})$  and that   $(\partial^2/ \partial \theta\partial \theta^\top)D_n(\theta_0, I_n)$ is invertible, 
   the following basic  expression is then  obtained,
   %we can write  
 %  \[ \sqrt{n}\big(\widehat{\theta}_n - \theta_0\big)  = -\Big(  \frac{\displaystyle \partial^2}{\displaystyle \partial \theta\partial \theta^\top}D(\theta_0, I_n)\Big)^{-1}\sqrt{n}\Big(\frac{\displaystyle \partial}{\displaystyle \partial \theta}D(\theta_0, I_n) -
  % \frac{\displaystyle \partial}{\displaystyle \partial \theta}D(\theta_0, f) \Big) + o_P(1).\]
%Since
%\[ \frac{\displaystyle \partial}{\displaystyle \partial \theta}D(\theta_0, I_n) -
   %\frac{\displaystyle \partial}{\displaystyle \partial \theta}D(\theta_0, f) =-\frac{1}{2\pi} \int_{-\pi}^{\pi}  \frac{1}{f^2_{\theta_0}(\lambda)}\frac{\displaystyle \partial }{\displaystyle\partial \theta} f_\theta(\lambda)\Big|_{\theta=\theta_0} %\big( I_n(\lambda) -f(\lambda)\big)d\lambda, \]
  % we obtain  the expression, 
   \begin{equation} \label{eq.theta-distr1}
   \sqrt{n}\big(\widehat{\theta}_n - \theta_0\big)  = \Big(  \frac{\displaystyle \partial^2}{\displaystyle \partial \theta\partial \theta^\top}D_n(\theta_0, I_n)\Big)^{-1}
   \sqrt{n}\int_{-\pi}^\pi g_{\theta_0}(\lambda)  \big( I_{n}(\lambda) - f(\lambda) \big) d\lambda + o_P(1).
   \end{equation} 
   Here  $g_{\theta_0}(\lambda) $ is a   $m$-dimensional vector of scores, the  $ j$th element of which is given by 
   \[ g_{j,\theta_0}(\lambda) = \frac{1}{2\pi f_{\theta_0} (\lambda)} \frac{\partial}{\partial \theta_j}\log f_{\theta_0}(\lambda)= - \frac{1}{2\pi} \frac{\partial}{\partial \theta_j}f^{-1}_{\theta_0}(\lambda),\]
   $j=1,2, \ldots, m$ and  $ f^{-1}_\theta = 1/f_\theta$.
  %Under certain assumptions, 
  Equation  (\ref{eq.theta-distr1})  suggests that 
  %can be    used  to derive  the limiting distribution of  $ \widehat{\theta}_n$.
  %; see  for instance  XXX REFERENCES and Taniguchi and Kakizawa (2000). 
 %As this expression shows,  
   the  distribution of  Whittle's estimator  $\widehat{\theta}_n$, can be well approximated  by  the product of the inverse of the $m\times m$  random  matrix $ W_n=\big( \partial^2/ \partial \theta_j\partial \theta_kD_n(\theta_0, I_n)\big)_{j,k=1,2, \ldots, m}$ with  the $m$-dimensional vector of  integrated periodograms 
   $    \big(\sqrt{n}\int_{-\pi}^\pi g_{j,\theta_0}(\lambda)  \big( I_{n}(\lambda) - f(\lambda) \big) d\lambda, j=1,2, \ldots, m\big)^\top$. 
   
  Based on expression (\ref{eq.theta-distr1}), asymptotic theory for Whittle estimators has been developed in the literature   under a variety of assumptions on the dependence structure of the underlying process.  In early papers Walker (1964) and Hannan (1973) derived asymptotic normality under linearity assumptions while Hosoya (1979) also allowed for long range dependence.   For  a  more  recent review of the related literature as well as  derivations of asymptotic normality based  on a physical dependence measure,   we refer to Shao (2010).  
%HERE REFERENCES and reference to  SHAO(Ecom., Theory paper) . 
    It is  typically shown, that,
    % under  some  assumptions on the smoothness of the spectral density $ f_\theta$ as well as on the  moment and on the dependence properties of the process $ \{X_t, t \in \Z\}$,  that, 
    %it  has    been  shown that,  as $ n \rightarrow \infty$,
   \begin{equation} \label{eq.theta-distr2}
      \sqrt{n}\big(\widehat{\theta}_n - \theta_0\big)   \stackrel{D}{\rightarrow} {\mathcal N}\big( 0, W^{-1}(V_1+V_2)W^{-1} \big),
   \end{equation}
   as $n\rightarrow\infty$, where  the $m\times m$ matrices $ W$, $V_1$ and $ V_2$ are      given by 
   \[ W=\Big( \frac{\partial^2}{\partial \theta_j\partial\theta_k} D(\theta_0,f)\Big)_{j,k=1,2, \ldots,m},\] 
   \[ V_1 = \Big( 4\pi\int_{-\pi}^\pi g_{j,\theta_0}(\lambda)g_{k,\theta_0}(\lambda) f^2(\lambda) d\lambda\Big)_{j,k=1,2, \ldots, m}\]
   and
   \[ V_2 =  \Big(2\pi \int_{-\pi}^\pi\int_{-\pi}^\pi   g_{j,\theta_0}(\lambda_1)g_{k,\theta_0}(\lambda_2) f_4(\lambda_1,\lambda_2, -\lambda_2) d\lambda_1d\lambda_2 \Big)_{j,k=1,2, \ldots, m}.\]

 Observe that 
 %$ V_1+V_2$ depends on the second and on the fourth order  structure of the underlying process  $\{X_t, t \in \Z\}$, while       
   $ W$ and $ V_1$  only depend on the parametric and the true spectral densities,   that is on $ f_{\theta_0}$ and $f$,  while the matrix $ V_2$ depends on  $ f_{\theta_0}$ and on the fourth order cumulant spectral density  $ f_4$ of the underlying process $ \{X_t, t \in \Z\}$. The latter is  defined as 
 \[ f_4(\lambda_1,\lambda_2,\lambda_3) =\frac{1}{(2\pi)^3} \sum_{h_1,h_2,h_3 \in \Z}cum(X_0,X_{h_1}, X_{h_2}, X_{h_3}) \exp^{-i(h_1\lambda_1 + h_2\lambda_2 + h_3\lambda_3)},   \]
 where $ cum(X_0,X_{h_1}, X_{h_2}, X_{h_3})$ denotes  the fourth order cumulant of the process $ \{X_t; t \in \Z\}$,  cf. Rosenblatt (1985).   Notice  that  if $ f=f_{\theta_0}$, that is,  if the spectral density of the process belongs to the parametric family $ {\mathcal F}_\theta$ and  
  the model is correctly specified, then  $V_1 = 2W$  and  the covariance matrix of 
 the limiting Gaussian distribution (\ref{eq.theta-distr2})  is given by $ W^{-1}(2I_m +V_2W^{-1})$, where  $I_m$ denotes  the $m\times m $ unit matrix.  

The matrix  $V_2$ in general does not   entirely disappear even if the underlying process is linear, that is,  if $ X_t$ is generated as 
$ X_t=\sum_{j=-\infty}^\infty \psi_j \varepsilon_{t-j}$, where  $ \sum_{j=-\infty}^\infty |\psi_j|<\infty$ and  the $\varepsilon_t$'s are zero mean, i.i.d. innovations with variance $\sigma^2>0$.
%class $ {\mathcal F}_\theta$ of parametric models considered  is linear, 
However, in this case,  the matrix  $V_2$ simplifies considerably. To elaborate, recall that under  the assumption of linearity and  of validity of Kolmogorov's formulae, (see Blockwell and Davis (1991), Ch. 5.8),
% the parameter vector $ \theta$ can be split in two parts, say  
$ \theta=(\sigma^2,\tau)$ where
% $ \sigma^2$ denoting the variance of the i.i.d. innovations $ \varepsilon_t$  driving the linear process and 
$\tau\in \Re^{m-1}$  is free of the innovation variance $\sigma^2$. 
%also contains 
% the variance  $ \sigma^2$ of the i.i.d. innovations $ \{\varepsilon_t\}$ driving the linear process.  However, 
The dependence on the fourth order moment structure of Whittle estimators disappears then if one is solely interested in the distribution of the estimators of the part  $ \tau$ of the parameter vector.
This is due to the fact  that for linear processes, the spectral density  $ f_\theta(\cdot) $ factorizes as $ f_\theta(\cdot) =h_\tau(\cdot) \sigma^2/(2\pi)$, where the function $h_\tau(\cdot)$ depends  on  $\tau$ only. Then, and since for  the same class of processes,  
 $ f_4(\lambda_1,\lambda_2, -\lambda_2) = (2\pi)^{-1} \eta_4 f(\lambda_1)f(\lambda_2)$,
 with $ \eta_{4}=E(\varepsilon^4_1/\sigma^4 -3)$, the rescaled fourth order cumulants (kurtosis)  of the i.i.d. innovations,     we get that, 
 \begin{align*}
 2\pi \int_{-\pi}^\pi\int_{-\pi}^\pi  & g_{j,\tau_0}(\lambda_1)g_{k,\tau_0}(\lambda_2) f_4(\lambda_1,\lambda_2, -\lambda_2) d\lambda_1d\lambda_2 \\
 & =   \eta_{4} \int_{-\pi}^\pi g_{j,\tau_0}(\lambda_1)f(\lambda_1)d\lambda_1 \int_{-\pi}^\pi g_{k,\tau_0}(\lambda_2) f(\lambda_2) \lambda_2  =0.
 \end{align*}
The  last equality follows since   the score function implies  $ \int_{-\pi}^\pi g_{s,\tau_0}(\lambda)f(\lambda)d\lambda=0$ for every $ s=1,2, \ldots, m-1$, where 
 $g_{s,\tau_0}(\cdot)$ denotes the partial derivative of  $ -h^{-1}_{\tau}(\cdot)/(2\pi)$ with respect to the $s$-th variable of the $m-1$ dimensional vector $ \tau$,  evaluated at $\tau=\tau_0$.
 Thus the $(m-1)\times(m-1)$ submatrix   of $V_2$ which corresponds to the elements of the vector $\tau$ only, consists of zeros.   
  Notice  that  this  simplification  does not  hold true for   the components  of the matrix $V_2$ which are affected by  the estimator of  $\sigma^2$; see  
 Dahlhaus and Janas (1996) for more details.  

% In any case, however,  the limiting distribution of  $\widehat{\theta}_n$ depends, apart from the second order characteristics, also  on the fourth order moment structure of the underlying process $ \{X_t, t\in \Z\}$.

Now, a  close look at the derivations leading to the  covariance formulae $ W^{-1}(V_1+V_2)W^{-1} $ of the limiting Gaussian distribution (\ref{eq.theta-distr2}),   reveals  that the term $ V_2$ is solely due to the weak and asymptotically vanishing  covariance of the periodogram ordinates across frequencies.  
Recall the basic expression of the covariance of the periodogram for nonzero Fourier frequencies $ |\lambda_{j,n}| \neq |\lambda_{k,n}|$, 
\begin{equation} \label{eq.CovPer}
{\rm cov}( I_n(\lambda_{j,n}), I_n(\lambda_{k,n}) )=   \frac{1}{n} f_4(\lambda_{j,n},\lambda_{k,n}, -\lambda_{k,n})(1+o(1)) +  O(n^{-2}).
\end{equation}
 Summing up  these   covariances  over all frequencies $|\lambda_{j,n} | \neq |\lambda_{k,n}| $ in the set $ {\mathcal G}(n)$, leads to  a non-vanishing contribution  to the limiting distribution of Whittle estimators as  expressed  by the matrix $V_2$.  As already mentioned,   frequency domain bootstrap procedures  which 
 generate independent periodogram ordinates, like for instance the multiplicative perodogram bootstrap,  can  not imitate the weak dependence structure  (\ref{eq.CovPer}). Therefore, and besides  the special case $ f_4=0$, e.g. Gaussian time series,  such procedures   do not  appropriately  capture  the term  $ V_2 $ and  they fail in  consistently  estimating  the distribution of $ \sqrt{n}(\widehat{\theta}_n-\theta_0)$.
  
\section{The basic bootstrap procedure} \label{sec.3}
 
 Our goal is to develop a consistent bootstrap  estimator 
 %interest is focused on  estimating  
of  the distribution of $ L_n=\sqrt{n}(\widehat{\theta}_n-\theta_0)$ without imposing  restrictive structural assumptions on the process $ \{X_t,t\in\Z\}$ and allowing  at the same time for the case of model misspecification. Toward this goal,  we propose a     hybrid, frequency domain bootstrap procedure which  builds upon 
  the multiplicative periodogram bootstrap and appropriately extends  it,  in order to overcome its limitations  for the class of Whittle estimators.  The  procedure consists of two main 
  parts. The first part is based on the  multiplicative bootstrap 
 approach as proposed by Franke and H\"ardle (1992) and Dahlhaus and Janas (1996). It is used to  estimate all features of the distribution of $L_n$ including  the parts of the covariance matrix     which do not depend on the fourth order characteristics of the process $ \{X_t; t \in \Z\}$ and in particular the matrices $W$ and $V_1$.
 % that is the matrices  $  W$ and $ V_1$. 
 This is done in Step 2 and Step 3 of the following algorithm.   However,    and as already mentioned, 
 since  the   multiplicative periodogram bootstrap  generates independent pseudo periodogram ordinates, it can not  be used to imitate 
 % capture 
 the fourth order characteristics of the process which affect the distribution of $L_n$ and more specifically the  matrix $ V_2$. 
  %This is so  since the latter characteristics   are due to the weak dependence of the periodogram ordinates among frequencies.
  The second part of our bootstrap procedure corrects for this shortcoming. This is achieved  by  generating pseudo periodograms  of subsamples of length $b$, $ b< n$, using   randomly selected  sets of appropriately defined frequency domain residuals.  The advantage of these pseudo periodograms  of subsamples  is that they retain  the weak dependence structure of the periodogram across  the Fourier frequencies corresponding to the subsamples. They  can, therefore, be used to consistently estimate the missing part $ V_2$.  This is done in Step 4 and Step 5 of the   bootstrap algorithm presented  bellow. Putting these two parts  together  in an appropriate way,  finally,  leads   in Step 6,  to a consistent estimator of the distribution of the random sequence $ \sqrt{n}(\widehat{\theta}_n-\theta_0)$ of interest.
  %,  including  the covariance matrix $ W^{-1}(V_1+V_2)W^{-1}$.   

  The following algorithm implements the above ideas  and is the basic  
     hybrid  frequency domain  bootstrap procedure proposed in this paper. 
     %Notice that this  bootstrap procedure  is   used  to estimate  the distribution  of $ L_n=\sqrt{\theta}_n-\theta_0)$ of interest. 
  % to estimate the distribution of interest.
%   can be used 
% in order which is 
% based on  a   hybrid bootstrap proposal 
%in the spirit of   Bootstrap Proposal~\ref{bo.HPB-specmeans}.   The following description  summarizes the steps used to estimate the distribution of  interest using such a hybrid  bootstrap proposal.
 \vspace*{0.25cm}
%\setdefaultleftmargin{0,75em}{2em}{}{}{}{}
%\begin{bootstrap}   ({\it  Hybrid Periodogram Bootstrap for  $ \widehat{\theta}_n $})
%\label{bo.Whittle1}
\begin{enumerate}
\item[{\it Step 1:}]%\hspace*{-0,75em} {\it Step 1:} 
\  Calculate Whittle's estimator $ \widehat{\theta}_n$.
% of $ \theta_0$.
\item[{\it Step 2:}]
%\hspace*{-0,75em} {\it Step 2:} 
\ Let $ \widehat{f}$ be a nonparametric  estimator of $ f$.  For $ j=1,2, \ldots, N$, generate  
$$ I^\ast_n(\lambda_{j,n})=\widehat{f}(\lambda_{j,n}) \cdot U^\ast_j,$$   where $ U_j^\ast$ are i.i.d. standard exponential distributed random variables. Set $  I^\ast_n(\lambda_{j,n})= I^\ast_n(-\lambda_{j,n})$ for $ j=-1,-2, \ldots, -N$. 
%$ j \in {\mathcal G}(n)$ as in  the multiplicative Bootstrap Proposal~\ref{bo.mul-per-boo}.
\item[{\it Step 3:}]
%\hspace*{-0,75em} {\it Step 3:}
 \ Let   
\[ \widehat{\theta}^\ast_n = \arg \min_{\theta\in\Theta}\, D_n(\theta, I^\ast_n)  \  \mbox{and} \   \widehat{\theta}_0 = \arg \min_{\theta\in\Theta}\, D_n(\theta, \widehat{f}).\]
Define
%\[ \widetilde{L}^\ast_n= \big(V^\ast_{1,n}\big)^{-1/2}  W^\ast_n \sqrt{n}(\widehat{\theta}^*_n -\widehat{\theta}_0),\]
%where 
\[ V_{1,n}^\ast = Var^\ast(M_n^\ast)  \ \  \mbox{and} \ \ W_n^\ast = \Big(\frac{\partial^2}{\partial \theta \partial \theta^\top}D_n(\theta, I^*_n) \Big|_{\theta=\widehat{\theta}_0} \Big),\]
where 
\[ M^\ast_n=\frac{2\pi}{\sqrt{n}}
   \sum_{j \in {\mathcal G}(n)} g_{\widehat{\theta}_0}(\lambda_{j,n})  \big( I_{n}^\ast(\lambda_{j,n}) - \widehat{f}(\lambda_{j,n}) \big),\]
   and  
 $g_{\widehat{\theta}_0}(\lambda)$ the $m$-dimensional vector
\[ g_{\widehat{\theta}_0}(\lambda) = \Big( g_{j,\widehat{\theta}_0}(\lambda) = -\frac{1}{2\pi}\frac{\partial}{\partial \theta_j} f^{-1}_\theta(\lambda)\Big|_{\theta=\widehat{\theta}_0}, \  j=1,2, \ldots, m\Big)^\top.\]
Calculate the pseudo random vector $Z_n^\ast$ defined by
\[ Z^\ast_n = \big(V^\ast_{1,n}\big)^{-1/2}  W^\ast_n \sqrt{n}(\widehat{\theta}^*_n -\widehat{\theta}_0).\]
%   Furthermore,  
%\[ . \]    
%Also calculate 
%\[ M_n^\ast =  \frac{2\pi}{\sqrt{n}} \sum_{j \in {\mathcal G}(n)}g_{\widehat{\theta}_n}(\lambda_{j,n}) \big(I_n^\ast(\lambda_{j,n}) -\widehat{f}(\lambda_{j,n}) \big) .\]
\item[{\it Step 4:}]%\hspace*{-0,75em} {\it Step 4:}  
Select an integer $ b <n$  and generate $k=[n/b]$  pseudo periodograms 
pseudo periodograms $   I_b^{(\ell)}(\lambda_{j,b})$, $ \ell=1,2, \ldots, k$, where  
\[ I_b^{(\ell)}(\lambda_{j,b})=\widehat{f}(\lambda_{j,b}) \cdot U^{(\ell)}_b(\lambda_{j,b}).\]
Here,    
$U^{(\ell)}_b(\lambda_{j,b})= I_b^{(\ell)}(\lambda_{j,b})/\widetilde{f}(\lambda_{j,b})$, where 
\[ \hspace*{1.5cm} I_b^{(\ell)}(\lambda_{j,b}) = \frac{1}{2\pi} \big|\sum_{s=1}^b X_{i_\ell +s-1} e^{-i s \lambda_{j,b}}\big|^2\]
and   $ i_\ell$, $ \ell =1,2, \ldots, k$,  are  i.i.d. random variables uniformly distributed on the set $ \{1,2, \ldots, n-b+1\}$.
Furthermore, 
\[   \widetilde{f}(\lambda_{j,b})=\frac{1}{n-b+1}\sum_{t=1}^{n-b+1} I_b^{(t)}(\lambda_{j,b}),\]
with $ I_b^{(t)}(\lambda)=(2\pi)^{-1}\big|\sum_{s=1}^{b}X_{t+s-1}e^{-it\lambda}\big|^2$.  
%
%
%
%
%
%\ Generate   pseudo random variables    $ V_n^+$  as 
%\[ V^+_n = \sqrt{kb} \frac{1}{k} \sum_{l=1}^{k} \frac{2\pi}{b} \sum_{j\in {\mathcal G}(b)} g_{\widehat{\theta}_0}(\lambda_{j,b}) \big( I_b^{(\ell)}(\lambda_{j,b}) - \widehat{f}(\lambda_{j,b})\big).\]
%In the above expression,  
% %$g_{\widehat{\theta}_n}(\lambda)$ is the $m$-dimensional vector
%%\[ g_{\widehat{\theta}_n}(\lambda) = \Big( g_{j,\widehat{\theta}_n}(\lambda) = \frac{1}{2\pi f^2_{\widehat{\theta}_n}(\lambda) }\frac{\partial }{\partial \theta} f_\theta(\lambda)\Big|_{\theta=\widehat{\theta}_n}, \  j=1,2, \ldots, m\Big)^\top\]
%%and 
%$   I_b^{(\ell)}(\lambda_{j,b})$, $ \ell=1,2, \ldots, k$, are  obtained  as  $ I_b^{(\ell)}(\lambda_{j,b})=\widehat{f}(\lambda_{j,b}) \cdot U^{(\ell)}_b(\lambda_{j,b})$. Here  
%$U^{(\ell)}_b(\lambda_{j,b})= I_b^{(\ell)}(\lambda_{j,b})/\widetilde{f}(\lambda_{j,b})$ where 
%\[ \hspace*{1.5cm} I_b^{(\ell)}(\lambda_{j,b}) = \frac{1}{2\pi} \big|\sum_{s=1}^b X_{i_\ell +s-1} e^{-i s \lambda_{j,b}}\big|^2 \ \mbox{and} \   \widetilde{f}(\lambda_{j,b})=\frac{1}{n-b+1}\sum_{t=1}^{n-b+1} I_b^{(t)}(\lambda_{j,b}).\]
%The   $ i_\ell$, $ \ell =1,2, \ldots, k$,  are  i.i.d. random variables uniformly distributed on the set $ \{1,2, \ldots, n-b+1\}$ while $ I_b^{(t)}(\lambda)=(2\pi)^{-1}\big|\sum_{s=1}^{b}X_{t+s-1}e^{-it\lambda}\big|^2$.  
%\ Generate   pseudo random variables    $ V_n^+$  as 
Calculate  the pseudo random variables $ M_n^+$ as  
\[ M^+_n = \sqrt{kb} \frac{1}{k} \sum_{l=1}^{k} \frac{2\pi}{b} \sum_{j\in {\mathcal G}(b)} g_{\widehat{\theta}_0}(\lambda_{j,b}) \big( I_b^{(\ell)}(\lambda_{j,b}) - \widehat{f}(\lambda_{j,b})\big).\]
\item[{\it Step 5:}] %\hspace*{-0,75em} {\it Step 5:} 
\ Calculate  the $m\times m$ matrix  $ V_{2,n}^+$ as
\[  V_{2,n}^{+} = \Sigma^+_n - C_n^+,\]
where 
%\[ W_n^\ast = \Big(\frac{\partial^2}{\partial \theta \partial \theta^\top}D(\theta, I^*_n) \Big|_{\theta=\theta^*} \Big), \] 
$ \Sigma^+_n = {\rm Var}^\ast(M_n^+) $,    
%\Sigma^\ast_n ={\rm Var}^\ast(L_n^\ast), \  \ \ \widetilde{\Sigma}^\ast_n = W_n^\ast \Sigma_n^\ast W_n^\ast    \]
%\[ \Sigma^+_n = {\rm Var}^+(V_n^+), \  \  \  \Sigma^\ast_n ={\rm Var}^\ast(L_n^\ast), \  \ \ \widetilde{\Sigma}^\ast_n = {\rm Var}^\ast(M^\ast_n)   \]
%\mbox{and} 
%
$   C_n^+ = \big(c_n^{(r,s)}\big)_{r,s=1,2, \ldots, m}$  and 
the  elements $ c_n^{(r,s)} $ 
%of the matrix $ C_n^+$  are 
are given by  
\begin{align*}
 c_n^{(r,s)} &  = \frac{8\pi^2}{b} \sum_{j\in {\mathcal G}(b)} g_{r,\widehat{\theta}_0}(\lambda_{j,b})g_{s,\widehat{\theta}_0}(\lambda_{j,b})\widehat{f}(\lambda_{j,b})^2 \\
 & \ \ \ \ \  \times \Big(\frac{1}{n-b+1}\sum_{t=1}^{n-b+1} \frac{I_{b}^{(t)}(\lambda_{j,b})^2}{\widetilde{f}_b(\lambda_{j,b})^2} -1\Big).
 \end{align*} 
\item[{\it Step 6:}]%\hspace*{-0,75em} {\it Step 6:} 
\ Appoximate the distribution of $ L_n=\sqrt{n}(\widehat{\theta}_n-\theta_0)$ by the distribution of 
%\[   \widetilde{L}_n^\ast = \big(W_n^\ast \big)^{-1} \big(\widetilde{\Sigma}^*_n + \big( \Sigma^+_n-C_n^+\big)\big)^{1/2} \big(\Sigma_n^\ast)^{-1/2} L_n^\ast.\]
%\[   L_n^\ast = \big(W_n^\ast \big)^{-1} \big(V_{1,n}^\ast + V_{2,n}^+ \big)^{1/2} \big(V^\ast_{1,n}\big)^{-1/2}  W^\ast_n \sqrt{n}(\widehat{\theta}^*_n -\widehat{\theta}_0).\]
\[   L_n^\ast = \big(W_n^\ast \big)^{-1} \big(V_{1,n}^\ast + V_{2,n}^+ \big)^{1/2} \cdot Z^\ast_n.\]
\end{enumerate}
%\end{bootstrap}

\vspace*{0.25cm} 

Several  aspects of the above bootstrap algorithm are clarified  by the following series of remarks and comments. 

\begin{remark}{~}
 \begin{enumerate}
 \item[(i)] \  Observe that  in order  to appropriately  capture the effect of model  misspecification, i.e., the case $ f \notin {\mathcal F}_\theta$, 
 a nonparametric estimator $ \widehat{f}$  is  used  in {\it Step 2} to generate the pseudo periodogram ordinates $ I^\ast_n$ and 
 not the estimated parametric spectral density $ f_{\widehat{\theta}_n}$.  
 \item[(ii)] \ The estimator $ \widehat{\theta}_0$    in Step 3 is defined in a way which  imitates the properties  of $ \theta_0$. This estimator also 
 delivers  the appropriate centering of the bootstrap estimator $ \widehat{\theta}_n^\ast$, that is  
  $ \sqrt{n}(\widehat{\theta}_n^\ast-\widehat{\theta}_0)$  is used  as a bootstrap analogue of $ \sqrt{n}(\widehat{\theta}_n- \theta_0)$. 
  %However, and as the proof of the asymptotic validity of the bootstrap procedure shows, the essential property of $\widehat{\theta}_0$ required  is its consistency, i.e.,  $ \widehat{\theta}_0 \stackrel{P}{\rightarrow} \theta_0$.   Since $ \widehat{\theta}_n \stackrel{P}{\rightarrow} \theta_0$ too, the asymptotic properties of the above bootstrap procedure will not be affected if  $ \widehat{\theta}_0$   is   replaced by the estimator $ \widehat{\theta}_n$.   
 \item[(iii)] \  As the proof of Theorem~\ref{th.bootWhittle} shows,
 \[  \sqrt{n}(\widehat{\theta}^\ast_n-\widehat{\theta}_0) \stackrel{D}{\rightarrow} {\mathcal N}(0,W^{-1}V_{1}W^{-1}),\]
  in probability.  Furthermore, Lemma~\ref{le.Appendix2} of Section~\ref{sec.7}, shows that  $W_n^\ast \stackrel{P}{\rightarrow} W  $ and $V_{1,n}^\ast \stackrel{P}{\rightarrow}  V_1 $.  These facts imply that     the limiting distribution of the standardized pseudo random variable $ Z^\ast_n=   \big(V^\ast_{1,n}\big)^{-1/2}  W^\ast_n \sqrt{n}(\widehat{\theta}^*_n -\widehat{\theta}_0)$  appearing at the end of Step 3, has covariance matrix   the $m\times m$  identity  matrix. 
   %That is,  $\widetilde{L}_n \stackrel{D}{\rightarrow} {\mathcal N}(0,{\bf I}_m)$ and, 
\end{enumerate} 
\end{remark}

\begin{remark}{~}
\begin{enumerate}
\item[(i)] \ In Step 4 the resampled periodograms of subsamples of length $b$,  i.e., $ I_b^{(\ell)}(\lambda_{j,b})$,  are obtained by using the same spectral density estimator $\widehat{f}$  as in Step 2 but evaluated  at the Fourier frequencies $\lambda_{j,b}$ corresponding to the length $b$ of the subsamples. Furthermore, in order to generate   $ I_b^{(\ell)}(\lambda_{j,b})$, $j=1,2, \ldots, B$, the estimated spectral density 
$\widehat{f}(\lambda_{j,b})$ 
 is multiplied with  the  entire  set of frequency domain  residuals, denoted by $ U_b^{(\ell)}(\lambda_{j,b})$, $j=1,2, \ldots, B$. Because these  residuals are obtained by rescaling the entire set of  periodogram ordinates of a subsample, they  retain for the Fourier frequencies $\lambda_{j,b}$ of the subsample, the weak dependence structure of the periodogram.   Notice that the  particular rescaling of these residuals applied,  ensures that  $ E^\ast(U^{(\ell)}_b(\lambda_{j,b}))=1$, that is $E^\ast(I^{(\ell)}_b(\lambda_{j,b})) =\widehat{f}(\lambda_{j,b})$. 
 \item[(ii)] \  As  Lemma~\ref{le.Appendix2} of Section~\ref{sec.7} shows,   
 %bellow shows, the following assertions are true,
 %\[  \widetilde{\Sigma}_n^\ast \stackrel{P}{\rightarrow} V_1  \ \ \mbox{and} \ \    \Sigma^+_n-C_n^+ \stackrel{P}{\rightarrow} V_2.\]
 % \[  V_{1,n}^\ast \stackrel{P}{\rightarrow} V_1  \ \ \mbox{and} \ \   
 $ V_{2,n}^+ \stackrel{P}{\rightarrow} V_2$. This  
  implies  that $ V_{1,n}^\ast + V_{2,n}^+ \stackrel{P}{\rightarrow} V_1+ V_2$, in probability. That is,   
the pseudo random variables $ M^+_n$ based on the convolved periodograms  generated in {\it Step 4} of the above algorithm, are solely used to estimate the part $ V_2$ of the covariance matrix of the distribution
of  $ L_n$.  As already mentioned, this part   can not be captured  by  the distribution of $ \sqrt{n}(\widehat{\theta}^*_n-\widehat{\theta}_0)$ generated in Step 3 due to the independence of the pseudo periodograms $ I^\ast_n(\lambda_{j,n})$ across the Fourier frequencies $\lambda_{j,n}$. 
%This is true since  the pseudo periodograms $ I^\ast_n(\lambda_{j,n})$ are  independent across the Fourier frequencies $\lambda_{j,n}$.
%All other characteristics of  the distribution of $ \sqrt{n}(\widehat{\theta}_n-\theta)$   are estimated by the bootstrap pseudo random variable $ \sqrt{n}(\widehat{\theta}^*_n-\theta^*)$   generated in   {\it Step 2} and {\it Step 3} of the   above bootstrap proposal. 
\item[(iii)] \ The statistic $ \widetilde{f}$ appearing in Step 4 is itself a nonparametric estimator of the spectral density $f$ the properties of which have  been investigated in the literature; see  Dahlhaus (1985) and the references therein. However, we use in this step the estimator $ \widehat{f}$ for generating the pseudo periodograms $ I_b^{(\ell)} $  in order to ensure that the same spectral density estimator is used  here as in Step 2  of the bootstrap procedure.  
\end{enumerate}
\end{remark}

\begin{remark} To understand the motivation behind the displayed equation in Step 6, notice that 
as  the proof of Theorem 4.1 shows  and by  Lemma 7.2 of Section~\ref{sec.7}, we have that $ Z^\ast_n\stackrel{D}{\rightarrow} {\mathcal N}(0,I_m)$, in probability. Furthermore and by the same lemma, it holds true that   
\[ (W_n^\ast)^{-1}(V_{1,n}^\ast + V_{2,n}^+)^{1/2} \stackrel{P}{\rightarrow} W^{-1}(V_1 + V_2)^{1/2}.\]
These results  imply  that $ L_n^\ast$ defined in Step 6 satisfies $ L_n^\ast \stackrel{D}{\rightarrow} {\mathcal N}(0, W^{-1}(V_1 + V_2)W^{-1}),$ 
in probability, which coincides with  the   limiting distribution of  the random sequence $ L_n=\sqrt{n}(\widehat{\theta}_n-\theta_0)$.
\end{remark}

\begin{remark}
 The $m\times m$ matrix $W_n^\ast$ containing the second order partial derivatives may be difficult to calculate in some situations.  In this case an additional step  can be included in the above procedure the aim  of which will be to directly estimate $W_n^\ast$. To elaborate,  
 let $  \Sigma^\ast_n={\rm Var}^\ast( \sqrt{n}\big(\widehat{\theta}_n^\ast -\widehat\theta_0)\big) $ and recall  the definition of 
 %we  first  calculate in Step 2 the integrated pseudo periodograms
% \begin{equation} \label{eq.Mstar}
%   M_n^\ast =  \frac{2\pi}{\sqrt{n}} \sum_{j \in {\mathcal G}(n)}g_{\widehat{\theta}_n}(\lambda_{j,n}) \big(I_n^\ast(\lambda_{j,n}) -\widehat{f}(\lambda_{j,n}) \big).
%   \end{equation}
  $V_{1,n}^\ast$ in Step 3.  
The matrix  $W^\ast_n$   can then be  estimated by 
 \begin{equation} \label{eq.Wstar}
  \widehat{W}_n^\ast =  V_{1,n}^\ast \big(\Sigma_n^\ast \cdot V_{1,n}^\ast \big)^{-1/2}.
  \end{equation}
%  where $ Var^\ast(M_n^\ast)$ denotes the variance of $ M_n^\ast$ given in (\ref{eq.Mstar}) and which can be evaluated by Monte Carlo.
By the property   $ \Sigma^\ast_n \stackrel{P}{\rightarrow} W^{-1}V_1 W^{-1} $,  we have by  Lemma~\ref{le.Appendix2} of Section~\ref{sec.7}, that 
\[  V^\ast_{1,n}  \big(\Sigma_n^\ast\cdot  V_{1,n}^\ast \big)^{-1/2} \stackrel{P}{\rightarrow} V_1\big(W^{-1}V_1 W^{-1}V_1\big)^{-1/2} = V_1\big(W^{-1} V_1\big)^{-1} = W.\]
That is, $\widehat{W}^\ast_n$ given in (\ref{eq.Wstar}) consistently  estimates  $W$ and can, therefore,  be used to replace $ W_n^\ast$  in the bootstrap algoritm. 
\end{remark}

\section{Bootstrap Validity}
\label{se.bootvalid}

In this section we establish the asymptotic validity of the bootstrap procedure  proposed. %Under standard assumptions, these  modifications do not affect the asymptotic properties of   Whittle estimators. 
Toward this goal    we need to   impose some conditions on the dependence structure of the  process  $ \{X_t,t\in \Z\}$ as well as on the smoothness properties of the functions  and of the  spectral densities involved. These conditions are  summarized in the following  assumptions. 

{\bf Assumption 1:} 
\begin{enumerate} 
\item[(i)] \  The process $ \{X_t, t \in \Z\}$  has mean zero, is eight-order stationary, i.e., the joint cumulants  up to eighth-order,
 $ {\rm cum}(X_t, X_{t+h_1}, \ldots, X_{t+h_7})$, do not depend on $t$ for any $ h_1, h_2, \ldots, h_7 \in {\mathbb Z}$. Furthermore, 
 $  \sum_{h\in \Z}|h||{\rm cum}(X_0,X_h)| <\infty$,  $ \inf_{\lambda\in[0,\pi]}f(\lambda) >0$, 
 \[ \sum_{h_1,h_2,h_3\in \Z}(|h_1|+|h_2|+|h_3|)|{\rm cum}(X_0,X_{h_1},X_{h_2}, X_{h_3})| <\infty,\]
 and 
 \[  \sum_{h_1,\ldots,h_7\in \Z}|{\rm cum}(X_0,X_{h_1},\ldots, X_{h_7})| <\infty.\] 
\item[(ii)] \  
%$\Theta$ is  a compact subset of $ \Re^m$, $ \theta_0 $ defined in (\ref{eq.theta0}) belongs to the interior of $ \Theta$,   is unique and  
The sequence of Whittle estimators $\{\sqrt{n}( \widehat{\theta}_n-\theta_0), n\in \N\}$, satisfies  (\ref{eq.theta-distr2}).
%$ \sqrt{n}(\widehat{\theta}_n-\theta_0) \stackrel{D}{\rightarrow} {\mathcal N}(0, W^{-1}(V_1+V_2) W^{-1})$. 
\end{enumerate}

The above assumptions on the dependence structure of the process $ \{X_t,t\in\Z\}$ are rather weak and cover a wide range of processes considered  in the literature; see among others, Rosenblatt (1985), Doukhan and Le\'on (1989), Wu and Shao (2004) for summability properties of cumulants for processes satisfying different weak dependent conditions. 
The requirement of eighth-order stationarity seems inavoidable taking into account the fact that our derivations include calculations of the variance of time averaged products of periodograms   of subsamples at different frequencies.   Notice that beyond the summability requirements  of  Assumption 1 and in order to be as flexible  as possible, we do not impose any  further conditions on the moment or on the dependence structure of the underlying  process. Instead, we  directly require that  the asymptotic normality of  the sequence $\{ \sqrt{n}(\widehat{\theta}_n-\theta_0), n \in \N\}$,  as   stated in Assumption 1(ii),  holds true.   This also covers a wide range of processes satisfying a variety of weak dependence conditions; see the  discussion before equation (\ref{eq.theta-distr2}) in Section 2.

{\bf Assumption 2:} 
\begin{enumerate}
\item[(i)] \  %The family of parametric spectral densities ${\mathcal F}_\theta$ is given by 
%\begin{equation} \label{param-f}
\[{\mathcal F}_\theta = \{ f_\theta, \theta \in \Theta, \inf _{\theta\in\Theta}\inf_{\lambda\in [-\pi,\pi]}f_\theta(\lambda)\geq  \delta >0\},\]
where $\Theta$ is  a compact subset of $ \Re^m$ and $ \theta_0 $ defined in (\ref{eq.theta0}) is unique and belongs to the interior of $ \Theta$. 
%Furthermore,  $ \theta_1\neq \theta_2$ implies   $f_{\theta_1}(\lambda) \neq f_{\theta_2}(\lambda)$  on a set of frequencies $\lambda$ with positive Lebesgue measure.
%\end{equation}
\item[(ii)] \ $D(\theta,f)$ is twice differentiable in $\theta \in \Theta$ under the integral sign. 
\item[(iii)] \  $f_\theta(\lambda)$ is continuous at any  $ (\lambda,\theta) \in [-\pi,\pi]\times \Theta$.
\item[(iv)] \ The first and second order partial derivatives of  $ f^{-1}_\theta(\cdot) = 1/f_\theta(\cdot)$ with respect to $\theta$  are continuous at any $(\lambda,\theta) \in  [-\pi,\pi]\times \Theta$.
\item[(v)] \ The $m\times m$  matrix $ W = \Big(\frac{\displaystyle \partial^2}{\displaystyle \partial \theta_j \partial \theta_k} D(\theta_0,f)\Big)_{j,k=1,\ldots,m} $  is non singular.
\end{enumerate} 

Assumption 2 specifies the  
conditions imposed on the family $ {\mathcal F}_\theta$ of parametric spectral densities considered. Assumption 2(i) requires that  the spectral densities  $ f_\theta \in {\mathcal F}_\theta$  
are   bounded from bellow  away from zero for all frequencies $\lambda \in [-\pi,\pi]$ and for all $ \theta \in \Theta$.
% Furthermore, model identifiability 
%is ensured by the second requirement of Assumption 2(i). T
This part of the assumption as well as the  smoothness properties of $ f_\theta$  imposed in  part (ii) to  part (v),   are standard and common in the literature; see Taniguchi (1987) and Dahlhaus and Wefelmeyer (1996).   

Our next assumption specifies the required consistency properties of the nonparametric spectral density estimator
$ \widehat{f}$  used to generate the pseudo periodograms  $I^\ast_{n}$ and $ I_b^{(\ell)}$. It is a standard requirement of uniform consistency.
% of the spectral density estimator $ \widehat{f}$. 

{\bf Assumption 3:}   \ The nonparametric spectral density estimator $ \widehat{f}$  satisfies
\[ \sup_{\lambda \in [-\pi,\pi]}\big| \widehat{f}(\lambda) - f(\lambda) \big|\stackrel{P}{\rightarrow} 0.\]

Our  last assumption   concerns  the rate at which the subsampling size $b$, involved in the generation of the convolved periodograms of subsamples,  is allowed  to increase to infinity with the sample size $n$  in order to ensure consistency of the estimator $ \widehat{V}_{2,n}^+$ used. 

{\bf Assumption 4:} \   $ b\rightarrow \infty$ as $ n \rightarrow \infty$ such that $ b^3/n \rightarrow 0$.

We  now state the main result of this paper which establishes consistency  of the bootstrap proposal  $ L_n^\ast$  defined in Step 6 of the basic bootstrap algorithm and used to estimate  the distribution of $ L_n$.
 
\begin{theorem} \label{th.bootWhittle}
Let Assumptions 1 to 4 be satisfied. 
%the conditions   XXX (z.B. of Theorem~\ref{theorem2} of MPK)  be satisfied and assume that $ \widehat{\theta}_n=\theta_0 + {\mathcal O}_P(n^{-1/2})$. 
Then,  as $n\rightarrow \infty$, 
\[ \sup_{x\in {\mathbb R}^m} \Big| P(L_n^\ast  \leq x\big|X_1,X_2, \ldots, X_n)  - P(L_n\leq x) \Big| \rightarrow 0,\]
 in probability, where $P(L_n^\ast  \leq \cdot \big|X_1,X_2, \ldots, X_n)$ denotes the distribution function of the  random variable $ L^\ast_n$ given the time series $ X_1, X_2, \ldots, X_n$.  
\end{theorem}

%Note that    the above   result has been established 
% under quite general  dependence  conditions   on  the underlying  process $\{X_t, t\in \Z\}$ generating the observed time series $ X_1, X_2, \ldots,  X_n$   (see Assumption 1) and  by allowing  at the same time,  
% that the spectral density $f$ of the process generating the time series observed,  
% does not necessarily belongs to the parametric class $ {\mathcal F}_\theta$.
% 
 \section{Incorporating Whittle Likelihood Modifications} \label{sec.4}

It has been observed that despite their nice properties, Whittle estimators may  behave,  in certain,  small samples  situations, inferior compared to  the exact,  time domain, 
 maximum likelihood estimators.  More specifically, Whittle  estimators may  be biased  in small samples 
 due  to   errors  inherited in Whittle's frequency domain approximation of the time domain  Gaussian maximum likelihood  or   in cases where 
  the spectral density  of the underlying process contains (strong) peaks or  the periodogram suffers from the well known blurring or aliasing effects. These drawbacks motivated many researchers to investigate   modifications 
 of the basic Whittle likelihood   in order to improve the  finite sample performance of the estimators.
   Dahlhaus (1988) proposed and investigated 
   the use of tapered periodograms,  while  Velasco and Robinson (2000) combined tapering  with differencing the time series before obtaining Whittle's estimators.  Sykulski et al. (2019) introduced  a de-biased Whittle likelihood to reduce leakage and blurring effects and more recently,  Subba  Rao and Yang (2020) proposed  the boundary corrected Whittle likelihood and also  combined this with tapering.  
 In this section we will propose modifications, respectively extensions, of   the basic  bootstrap algorithm presented  in Section 3,  which appropriately take into account  such  modifications of  Whittle estimators. 
% Notice that  the different adaptions  of the basic bootstrap   algorithm  discussed in this section,  do not essentially affect the limiting distribution of the Whittle estimators  and, therefore,   the  established validity of the bootstrap. 

\subsection{Tapered Periodograms}
Applying a data  taper to the time series observed, leads to   the replacement of  the periodogram $ I_n(\lambda_{j,n})$ used in Whittle's likelihood approximation (\ref{eq.whittle-lk-disc})  by a  tapered periodogram, denoted  by $ I_{n,T}(\lambda_{j,n})$.
 Let   $ h_{t,n} = h(t/n)$ be  a data taper, that is  $ h : \Re\rightarrow [0,1]$ is  a function of bounded variation satisfying  $h(x)=0$ for $ x \notin (0,1]$. For any frequency $ \lambda_{j,n} \in {\mathcal F}_n$, the tapered periodogram is then defined as 
 $$ I_{n,T}(\lambda_{j,n}) = \frac{1}{ 2\pi H_{2,n}}J_{n,T}(\lambda_{j,n}) \cdot J_{n,T}(-\lambda_{j,n}),$$ 
 where 
 $ J_{n,T}(\lambda_{j,n}) = \sum_{t=1}^n h_{t,n}  X_t \exp\{-i\lambda_{j,n} t\}$ is the  finite Fourier transform of the tapered  time series.  Here and   for  any  $k \in \N$, $ H_{k,n} = \sum_{t=1}^n h_{t,n}^k$. 
% With respect to the asymptotic   distribution of Whittle's estimators stated in (\ref{eq.theta-distr2}), data tapering leads to 
% \[  \sqrt{n}\big(\widehat{\theta}_n - \theta_0\big)   \stackrel{D}{\rightarrow} {\mathcal N}\big( 0, \lim_{n\rightarrow\infty}\frac{nH_{4,n}}{H^2_{2,n}} W^{-1}(V_1+V_2)W^{-1} \big),\]
% see Dahlhaus (1988). The above expression for the variance of the limiting Gaussian distribution  coincides with (\ref{eq.theta-distr2}) in case $\lim_{n\rightarrow\infty}(nH_{4,n})/(H^2_{2,n})=1 $, which also is 
% described as that of an asymptotically vanishing data taper. 
 
To  incorporate  data tapering in  the bootstrap procedure presented in Section 3,  the pseudo periodograms $ I_{n}^\ast$ and $ I_{b}^{(\ell)}$  have to be replaced by tapered versions,  denoted  by $ I_{n,T}^\ast $ and $  I_{b,T}^{(\ell)}$, respectively.  
 
 The  tapered version $  I_{b,T}^{(\ell)}$ of the periodogram $ I_{b}^{(\ell)}$, $\ell =1,2, \ldots, k$,   based on subsamples of length $b$ and used in Step 4 and Step 5,  can be generated     as
  $$ I_{b,T}^{(\ell)}(\lambda_{j,b}) =\widehat{f}(\lambda_{j,b})\cdot   U_{b,T}^{(\ell)}(\lambda_{j,b}),$$
  where $ U_{b,T}^{(t)}(\lambda_{j,b})= I_{b,T}^{(t)}(\lambda_{j,b}) /\widetilde{f}_T(\lambda_{j,b})$ for $ t=1,2, \ldots, n-b+1$. In the last  expression,  $  \widetilde{f}_T(\lambda_{j,b}) = \sum_{t=1}^{n-b+1} I^{(t)}_{b,T}(\lambda_{j,b})/(n-b+1)$  and 
 $ I_{b,T}^{(t)}(\lambda_{j,b}) = J_{b,T}^{(t)}(\lambda_{j,b}) \cdot J_{b,T}^{(t)}(-\lambda_{j,b}) /(2\pi H_{2,b})$ with $J_{b,T}^{(t)}(\lambda_{j,b})$ the finite Fourier transform of the tapered subsample $ X_{t}, X_{t+1}, \ldots, X_{t+b-1}$, that is,  
 \[  J_{b,T}^{(t)}(\lambda_{j,b}) = \sum_{s=1}^b h_{s,b}  X_{t+s-1} e^{-i\lambda_{j,b} s}.\]

 The generation of the  tapered pseudo periodogram $ I_{n,T}^\ast$   used to replace $I^\ast_n$ in Step 2 and Step 3 is more involved. For this purpose, the following procedure  consisting of Step 2a) to Step 2c)  can be  used   to generate the tapered pseudo periodograms $ I_{n,T}^\ast$  and  to  replace  Step 2 of the basic bootstrap algorithm presented in Section 3 .
  \vspace*{0.25cm}
%\setdefaultleftmargin{0,75em}{2em}{}{}{}{}
%\begin{bootstrap}   ({\it  Hybrid Periodogram Bootstrap for  $ \widehat{\theta}_n $})
%\label{bo.Whittle1}
\begin{enumerate}
\item[{\it Step 2a:}] 
\  Generate $ \varepsilon_1^\ast, \varepsilon_2^\ast, \ldots, \varepsilon_n^\ast$ i.i.d., $ {\mathcal N}(0,1)$ distributed random variables and calculate 
for $ \lambda_{s,n}\in{\mathcal F}_n$, the normalized  finite Fourier transform,  
\[  Z^\ast_{s,n}= \frac{1}{\sqrt{2\pi n}} \sum_{t=1}^n \varepsilon_t^\ast e^{-i t \lambda_{s,n}}.\]
%Set   $ Z^\ast_{s,n}=\overline{Z}^\ast_{s,n}$ for $ \lambda_{s,n}<0$.    
\item[{\it Step 2b:}] \ For $t=1,2, \ldots, n$, calculate  the pseudo random variables
\[ X_t^\ast = \sqrt{\frac{2\pi}{n}} \sum_{\lambda_{s,n}\in {\mathcal F}_n}\widehat{f}^{1/2}(\lambda_{s,n}) Z^\ast_{s,n} e^{i t \lambda_{s,n}}.\]
\item[{\it Step 2c:}] \   For $ \lambda_{j,n}\in {\mathcal F}_n$, calculate  the finite Fourier transform of the tapered  pseudo time series $ X_1^\ast, X^\ast_2, \ldots, X_n^\ast$, that is,
\[ J_{n,T}^\ast(\lambda_{j,n}) = \sum_{t=1}^{n} h_{t,n} X_t^\ast e^{-it \lambda_{j,n}}.\]
%For $\lambda_{j,n} <0$ set  $J_{n,T}^\ast(\lambda_{j,n})=\overline{J}_{n,T}^\ast(\lambda_{j,n}) $.   
%\item[{\it Step 2d:}]  \ 
The tapered pseudo periodogram $ I^\ast_{n,T}(\lambda_{j,n})$  is then defined as 
\[ I^\ast_{n,T} (\lambda_{j,n}) = \frac{1}{2\pi H_{2,n}} J_{n,T}^\ast(\lambda_{j,n})\cdot \overline{J}_{n,T}^\ast(\lambda_{j,n}).\]
\end{enumerate}

\vspace*{0.25cm} 

\begin{remark}
Notice  that  if  we set   $ h_{t,n} \equiv 1$ for $ t=1,2, \ldots, n$, in  Step 2c) of the above algorithm,  which  corresponds to  the case of no taper, then we get    
\[  I^\ast_{n,T} (\lambda_{j,n}) = \widehat{f}(\lambda_{j,n}) |Z^\ast_{j,n}|^2.\]
The random variables $ |Z_{j,n}^\ast|^2$, $j=1,2, \ldots, N$,   are independent, which implies that  the periodogram ordinates   $ I^\ast_{n,T} (\lambda_{j,n})$ are independent across the  frequencies $\lambda_{j,n}$, $ j=1,2 ,\ldots, N$.  Furthermore and  since the $ |Z_{j,n}^\ast|^2 $ 
also have   a  standard  exponential distribution  for every   $ j $, we get that if $h_{t,n}\equiv 1$ for all $ t =1,2, \ldots, n$,   then the  
pseudo  periodograms $ I^\ast_{n,T} (\lambda_{j,n})$ generated following  Step 2a) to Step 2c)  have exactly the same  properties as the (non tapered) pseudo periodograms 
$ I_n^\ast(\lambda_{j,n})$  generated in Step 2 of the bootstrap algorithm 
presented in Section 3.  
\end{remark}

Now, the   tapered pseudo periodograms $I_{b,T}^{(\ell)} $ and $ I_{n,T}^\ast$  can be used  in the bootstrap algorithm to approximate the distribution 
of $ \sqrt{n}(\widehat{\theta}_{n,T}-\theta_{0})$, where $ \widehat{\theta}_{n,T}=\arg\min_{\theta\in \Theta}D_n(\theta, I_{n,T})$ and  $ D_n(\cdot) $ the function  given in (\ref{eq.whittle-lk-disc}). To elaborate, $ \widehat{\theta}_n$ appearing in Step 1 and elsewhere in this algorithm is replaced by the tapered estimator  $ \widehat{\theta}_{n,T}$. 
In Step 2 the bootstrap periodogram $ I^\ast_n$ is replaced by $ I^\ast_{n,T}$ and consequently $ \widehat{\theta}_n^\ast$ in Step 3  by 
\[ \widehat{\theta}_{n,T}^\ast = \arg\min_{\theta \in \Theta} D_n(\theta, I^\ast_{n,T}).\]
Thus the bootstrap approximation of $  \sqrt{n}(\widehat{\theta}_{n,T}-\theta_{0})$  in the same step is  given by  
$ \widetilde{L}_{n}^\ast = \sqrt{n}(\widehat{\theta}^\ast_{n,T}-\widehat{\theta}_0)$. Furthermore,   the  tapered pseudo periodogram $ I^\ast_{n,T}$ is used in the expression   of the  vector $ M_n^\ast$ in Step 3, while  the matrix 
$ W^\ast_n$   in the same step  is  calculated using  $ D_n(\theta, I^\ast_{n,T})$, that is  $ W_n^\ast$  is replaced by 
$$W^\ast_{n,T}=  \Big(\frac{\partial^2}{\partial \theta \partial \theta^\top}D_n(\theta, I^*_{n,T}) \Big|_{\theta=\widehat{\theta}_0} \Big).$$ 
Finally, in Step 4, the periodogram of the random subsamples $ I_b^{(\ell)}$   in the expression for $ M_n^+$ 
as well as  the periodogram $ I_b^{(t)}$ in the same step and  in Step 5, are  replaced by their  tapered versions $ I_{ b,T}^{(\ell)}$ and $ I_{b,T}^{(t)} $, respectively.  
%and $ D_n(\thetra,I^\ast_n) $ in Step 5 by $ D_n(\theta, I_{n,T})$. 

\subsection{Debiased Whittle Likelihood}

Debiasing the Whittle likelihood  has been proposed   by Sykulski et al. (2019). The basic idea is to replace the parametric spectral density $f_\theta$  appearing in $D_n(\theta,I_n )$ by a smoothed  version  which equals   the expectation of the periodogram $I_n(\lambda)$ under the assumption that  $ f=f_\theta$. More specifically,  the objective function  considered by this modification  of Whittle's  approximation of the quasi Gaussian likelihood is given by 
\begin{equation} \label{eq. WL-debiased}
D_n^{(db)}(\theta, I_n) =   \frac{1}{n}\sum_{j\in {\mathcal G}(n)}\Big\{\log \overline{f}_\theta (\lambda_{j,n}) + \frac{\displaystyle I_n(\lambda_{j,n})}{\displaystyle \overline{f}_\theta(\lambda_{j,n})}\Big\},
\end{equation}
where 
\[ \overline{f}_\theta (\lambda_{j,n}) = \int_{-\pi}^\pi K_n(w-\lambda_{j,n})f_\theta(\omega)d\omega \]
and $ K_n(\cdot)$ is the Fejer-kernel, 
$$ K_n(x) = {\bf 1}_{\{x=0\}} n/2\pi +  {\bf 1}_{\{x\neq 0\}} \sin^2(nx/2)\big/(2\pi n \sin^2(x/2)\big).$$
Notice that  if $ f=f_\theta$,  then $ {\rm E}(I_{n}(\lambda_{j,n})) = \overline{f}_\theta(\lambda_{j,n})$; see  for instance, Rosenblatt (1963).  That is,  in this case, the periodogram $ I_n$ apearing in $ D^{(db)}(\theta,I_n)$   is an unbiased estimator of $ \overline{f}_\theta$,  which justifies the name given to this  modification. 

%Since 
%$ \sup_{\lambda \in [-\pi,\pi]} |\overline{f}_\theta(\lambda) - f_\theta(\lambda)| \rightarrow 0$ as $n\rightarrow \infty$, the large sample properties  of the  de-biased estimator 
%$ \widetilde{\theta}_n=  \arg \min_{\theta\in\Theta}D_n^{(db)}(\theta, I_n)$ are the same to those of $ \widehat{\theta}_n$. Nevertheless,  
In 
the following steps we summarize the modifications needed in order to adapt the basic bootstrap algorithm presented in Section 3 to imitate  the random  properties   of the estimator $ \widetilde{\theta}_n$.

 \vspace*{0.25cm}
\begin{enumerate}
\item[] \noindent{\it Step I:}
\  Calculate $ \widetilde{\theta}_n= \arg \min_{\theta\in\Theta}D_n^{(db)}(\theta, I_n)$.
\item[] \noindent{\it Step II:} \ The same as Step 2 of the bootstrap algorithm in Section 3.
\item[] \noindent{\it Step III:} \ Calculate  
$ \widetilde{\theta}^\ast_n = \arg \min_{\theta\in\Theta}\, D^{(db)}_n(\theta, I^\ast_n)$, \ \   $ \widetilde{\theta}_0 = \arg \min_{\theta\in\Theta}\, D^{(db)}_n(\theta, \widehat{f})$ 
and  $ \widetilde{V}_{1,n}^\ast = Var^\ast(\widetilde{M}_n^\ast) $, 
where 
\[ \widetilde{M}^\ast_n=\frac{2\pi}{\sqrt{n}}
   \sum_{j \in {\mathcal G}(n)} \overline{g}_{\widetilde{\theta}_0}(\lambda_{j,n})  \big( I_{n}^\ast(\lambda_{j,n}) - \widehat{f}(\lambda_{j,n}) \big)\]
   and  
 $\overline{g}_{\widetilde{\theta}_0}(\lambda)$ is the $m$-dimensional vector
\[ \overline{g}_{\widehat{\theta}_0}(\lambda) = \Big( \overline{g}_{j,\widehat{\theta}_0}(\lambda) = -\frac{1}{2\pi}\frac{\partial}{\partial \theta_j} \overline{f}^{-1}_\theta(\lambda)\Big|_{\theta=\widetilde{\theta}_0}, \  j=1,2, \ldots, m\Big)^\top.\]
Calculate
\[ \widetilde{W}_n^\ast = \Big(\frac{\partial^2}{\partial \theta \partial \theta^\top}D^{(db)}_n(\theta, I^*_n) \Big|_{\theta=\widetilde{\theta}_0} \Big)\]
and the  pseudo random vector $\widetilde{Z}_n^\ast$ defined by
\[ \widetilde{Z}^\ast_n = \big(\widetilde{V}^\ast_{1,n}\big)^{-1/2}  \widetilde{W}^\ast_n \sqrt{n}(\widetilde{\theta}^*_n -\widetilde{\theta}_0).\]
\item[] \noindent{\it Step IV:} \ The same as Step 4 of the bootstrap algorithm in Section 3 but  by replacing $ M_n^+$ by 
\[\widetilde{M}^+_n = \sqrt{kb} \frac{1}{k} \sum_{l=1}^{k} \frac{2\pi}{b} \sum_{j\in {\mathcal G}(b)} \overline{g}_{\widetilde{\theta}_0}(\lambda_{j,b}) \big( I_b^{(\ell)}(\lambda_{j,b}) - \widehat{f}(\lambda_{j,b})\big).\]
\item[] \noindent {\it Step V:}
\ Calculate  the $m\times m$ matrix   $ \widetilde{V}_{2,n}^{+} = \widetilde{\Sigma}^+_n - \widetilde{C}_n^+$, 
where 
$\widetilde{\Sigma}^+_n = {\rm Var}^\ast(\widetilde{M}_n^+) $,    
$   \widetilde{C}_n^+ = \big(\widetilde{c}_n^{(r,s)}\big)_{r,s=1,2, \ldots, m}$  and 
the  elements $ \widetilde{c}_n^{(r,s)} $ 
%of the matrix $ C_n^+$  are 
given by  
\begin{align*}
 \widetilde{c}_n^{(r,s)} &  = \frac{8\pi^2}{b} \sum_{j\in {\mathcal G}(b)} \overline{g}_{r,\widetilde{\theta}_0}(\lambda_{j,b})\overline{g}_{s,\widetilde{\theta}_0}(\lambda_{j,b})\widehat{f}(\lambda_{j,b})^2 \\
 & \ \ \ \ \  \times \Big(\frac{1}{n-b+1}\sum_{t=1}^{n-b+1} \frac{I_{b}^{(t)}(\lambda_{j,b})^2}{\widetilde{f}_b(\lambda_{j,b})^2} -1\Big).
 \end{align*} 
\item[] \noindent{\it Step VI:}
\ Appoximate the distribution of $ L_n=\sqrt{n}(\widetilde{\theta}_n-\theta_0)$ by the distribution of 
\[   \widetilde{L}_n^\ast = \big(\widetilde{W}_n^\ast \big)^{-1} \big(\widetilde{V}_{1,n}^\ast +\widetilde{V}_{2,n}^+ \big)^{1/2} \cdot \widetilde{Z}^\ast_n.\]
\end{enumerate}

%\vspace*{0.25cm} 

\subsection{Boundary Corrected Whittle Likelihood}

In order to reduce the  bias of Whittle estimators  caused by  boundary effects, 
% and in particular to the omission of linear predictors of the time series at hand outside the domain of observation, 
Subba Rao and Yang (2021) proposed the so called, boundary corrected Whittle likelihood. The objective function to be minimized in this case  is given by 
\begin{equation} \label{eq. WL-BoundCorr}
D_n^{(bc)}(\theta, \widetilde{I}_n) =   \frac{1}{n}\sum_{j\in {\mathcal G}(n)}\Big\{\log  f_\theta (\lambda_{j,n}) + \frac{\displaystyle \widetilde{I}_n(\lambda_{j,n})}{\displaystyle f_\theta(\lambda_{j,n})}\Big\},
\end{equation}
where  $ \widetilde{I}_n (\lambda_{j,n})= (2\pi n)^{-1}\widetilde{J}_n(\lambda_{j,n}) \overline{J}_n(\lambda_{j,n})$. Here,  $ J_n(\lambda_{j,n})=\sum_{t=1}^n X_t e^{-i \lambda_{j,n} t}$,  is the finite Fourier transform of the time series $ X_1, X_2, \ldots, X_n$  while $ \widetilde{J}_n(\lambda_{j,n})$ is  the boundary extended finite Fourier transform given by  $ \widetilde{J}_n(\lambda_{j,n}) = J_n(\lambda_{j,n}) + \widehat{J}_n(\lambda_{j,n})$.  The  boundary extended finite Fourier transform  $ \widetilde{J}_n(\lambda_{j,n})$   is obtained by calculating the finite Fourier transform of the out of sample extended time series 
\begin{equation} \label{eq.ExtTs}
 \ldots , \widehat{X}_{-1}, \widehat{X}_0, X_1, X_2, \ldots, X_n, \widehat{X}_{n+1}, \widehat{X}_{n+2}, \ldots ,
\end{equation}
 where the  pseudo observations $ \widehat{X}_t$ are the  best  linear predictors of the corresponding (not observed) values  $X_t$ based on an  AR(p) model.  To elaborate, let 
 $(\widehat{\phi}_{s,p}, s=1,2, \ldots, p)^\top$ be  the vector of Yule-Walker estimators obtained by fitting an AR(p) model to the time series $ X_1 X_2, \dots, X_n$. Then
 $ \widehat{X}_{n+s} =\sum_{j=1}^p \widehat{\phi}_{j,p} \widetilde{X}_{n+s-j}$ with $ \widetilde{X}_t=X_t$ if $ t \in \{1,2, \ldots, n\}$ and  $ \widetilde{X}_t=\widehat{X}_t$ if $t>n$.
 An analogue  expression yields for   $ \widehat{X}_{-s}$, $ s \leq 0$.  
 %Substituting recursively for  the predicted values $ \widehat{X}_t$ and calculating the finite Fourier transform  
 % of the time series (\ref{eq.ExtTs}), 
 This leads to the  expression  $ \widetilde{J}_n(\lambda_{j,n}) = J_n(\lambda_{j,n}) + \widehat{J}_n(\lambda_{j,n})$, where 
 the  ``extension"  part $\widehat{J}_n(\lambda_{j,n})$   can be written as 
\begin{align*}
\widehat{J}_n(\lambda_{j,n}) & = \frac{1}{ \widehat{\phi}_p(\lambda_{j,n})} \sum_{\ell =1}^p X_\ell \sum_{s=0}^{p-\ell} \widehat{\phi}_{\ell+s,p}e^{-i s \lambda_{j,n}} \\
& \ \ \ \ \  + e^{i n \lambda_{j,n}}\frac{1}{\overline{\widehat{\phi}_p(\lambda_{j,n})}} \sum_{\ell=1}^p X_{n+1-\ell}\sum_{s=0}^{p-\ell} \widehat{\phi}_{\ell+s,p} e^{i(s+1)\lambda_{j,n}},
\end{align*}
with $ \widehat{\phi}_p(\lambda) = 1-\sum_{s=1}^p \widehat{\phi}_{s,p} e^{-i s \lambda} $; see  Subba Rao and Yang (2021).

In the following we only describe how to modify the first part of the basic bootstrap algorithm, that is Step 1 to Step 3,  in order to get replicates of the boundary corrected Whittle estimators $ \breve{\theta}_n =  \arg \min_{\theta\in\Theta}\, D^{(bc)}_n(\theta, \widetilde{I}_n)$. The modifications needed for the second part of the basic bootstrap procedure which uses   convolved periodograms of subsamples (Step 4 to Step 5), easily follow from those  presented for  the first part.

%\vspace*{0.25cm}
%%\setdefaultleftmargin{0,75em}{2em}{}{}{}{}
%%\begin{bootstrap}   ({\it  Hybrid Periodogram Bootstrap for  $ \widehat{\theta}_n $})
%%\label{bo.Whittle1}
%\begin{enumerate}
%%\setlength{\itemsep}{3pt}
%\item[{\it Step 1a:}] 
%\  Generate $ \varepsilon_1^\ast, \varepsilon_2^\ast, \ldots, \varepsilon_n^\ast$ i.i.d., $ {\mathcal N}(0,1)$ distributed random variables and calculate 
%for $ \lambda_{s,n}\in{\mathcal F}_n$, the finite Fourier transform,  
%\[  Z^\ast_{s,n}= \frac{1}{\sqrt{2\pi n}} \sum_{t=1}^n \varepsilon_t^\ast e^{-i t \lambda_{s,n}}.\]
%%Set   $ Z^\ast_{s,n}=\overline{Z}^\ast_{s,n}$ for $ \lambda_{s,n}<0$.    
%\item[{\it Step 2b:}] \ For $t=1,2, \ldots, n$, calculate  the pseudo random variables
%\[ X_t^\ast = \sqrt{\frac{2\pi}{n}} \sum_{\lambda_{s,n}\in {\mathcal F}_n}\widehat{f}^{1/2}(\lambda_{s,n}) Z^\ast_{s,n} e^{i t \lambda_{s,n}}.\]
%\item[{\it Step 2c:}] \   For $ \lambda_{j,n}\in {\mathcal F}_n$, calculate  the finite Fourier transform of the tapered  pseudo time series $ X_1^\ast, X^\ast_2, \ldots, X_n^\ast$, that is,
%\[ J_{n,T}^\ast(\lambda_{j,n}) = \sum_{t=1}^{n} h_{t,n} X_t^\ast e^{-it \lambda_{j,n}}.\]
%%For $\lambda_{j,n} <0$ set  $J_{n,T}^\ast(\lambda_{j,n})=\overline{J}_{n,T}^\ast(\lambda_{j,n}) $.   
%%\item[{\it Step 2d:}]  \ 
%The tapered pseudo periodogram $ I^\ast_{n,T}(\lambda_{j,n})$  is then defined as 
%\[ I^\ast_{n,T} (\lambda_{j,n}) = \frac{1}{2\pi H_{2,n}} J_{n,T}^\ast(\lambda_{j,n})\cdot \overline{J}_{n,T}^\ast(\lambda_{j,n}).\]
%\end{enumerate}

\vspace*{0.25cm} 

\begin{enumerate}
\item[{\it Step 1$^{'}$:}]%\hspace*{-0,75em} {\it Step 1:} 
\  Calculate Whittle's estimator $ \breve{\theta}_n= \arg \min_{\theta\in\Theta}\, D^{(bc)}_n(\theta, \widetilde{I}_n)$.
% of $ \theta_0$.
\item[{\it Step 2$^{'}$:}]
%\hspace*{-0,75em} {\it Step 2:} 
\  For $t=1,2, \ldots, n$, calculate  the pseudo random variables
\[ X_t^\ast = \sqrt{\frac{2\pi}{n}} \sum_{\lambda_{s,n}\in {\mathcal F}_n}\widehat{f}^{1/2}(\lambda_{s,n}) Z^\ast_{s,n} e^{i t \lambda_{s,n}},\]
where 
\[  Z^\ast_{s,n}= \frac{1}{\sqrt{2\pi n}} \sum_{t=1}^n \varepsilon_t^\ast e^{-i t \lambda_{s,n}}\]
and $ \varepsilon_1^\ast, \varepsilon_2^\ast, \ldots, \varepsilon_n^\ast$ are  i.i.d., $ {\mathcal N}(0,1)$ distributed random variables.
Calculate
\[ \widetilde{J}_n^\ast(\lambda_{j,n}) = J^\ast_n(\lambda_{j,n}) + \widehat{J}_n^\ast(\lambda_{j,n}) ,\]
where $ J^\ast_{n}(\lambda_{j,n}) = \sum_{t=1}^n   X^\ast_t \exp\{-i\lambda_{j,n} t\}$ and 
\begin{align*}
\widehat{J}^\ast_n(\lambda_{j,n}) & = \frac{1}{\widehat{\phi}^\ast_p(\lambda_{j,n})} \sum_{\ell =1}^p X^\ast_\ell \sum_{s=0}^{p-\ell} \widehat{\phi}^\ast_{\ell+s,p}e^{-i s \lambda_{j,n}} \\
& \ \ \ \ \  + e^{i n \lambda_{j,n}}\frac{1}{\overline{\phi^\ast_p(\lambda_{j,})}} \sum_{\ell=1}^p X^\ast_{n+1-\ell}\sum_{s=0}^{p-\ell} \widehat{\phi}^\ast_{\ell+s,p} e^{i(s+1)\lambda_{j,n}}.
\end{align*}
In the above expression, $ \widehat{\phi}^\ast_p(\lambda) = 1-\sum_{s=1}^p \widehat{\phi}^\ast_{s,p} e^{-i s \lambda} $  and 
 $(\widehat{\phi}^\ast_{s,p}, s=1,2, \ldots, p)^\top$   is the vector of Yule-Walker estimators obtained by fitting an AR(p) model to the pseudo time series $ X_1^\ast, X_2^\ast, \dots, X_n^\ast$. Define,
 \[ \widetilde{I}_n^\ast(\lambda_{j,n}) = \frac{1}{2\pi n} \widetilde{J}_n^\ast(\lambda_{j,n}) \overline{J_n^\ast}(\lambda_{j,n}).\]
\item[{\it Step 3$^{'}$:}]
%\hspace*{-0,75em} {\it Step 3:}
 \ Let   
\[ \breve{\theta}^\ast_n = \arg \min_{\theta\in\Theta}\, D_n(\theta, \widetilde{I}^\ast_n).\]
%  \  \mbox{and} \   \widehat{\theta}_0 = \arg \min_{\theta\in\Theta}\, D_n(\theta, \widehat{f}).\]
Define
%\[ \widetilde{L}^\ast_n= \big(V^\ast_{1,n}\big)^{-1/2}  W^\ast_n \sqrt{n}(\widehat{\theta}^*_n -\widehat{\theta}_0),\]
%where 
\[ \widetilde{V}_{1,n}^\ast = Var^\ast(\widetilde{M}_n^\ast)  \ \  \mbox{and} \ \ \widetilde{W}_n^\ast = \Big(\frac{\partial^2}{\partial \theta \partial \theta^\top}D_n(\theta, \widetilde{I}^*_n) \Big|_{\theta=\widehat{\theta}_0} \Big),\]
where 
\[\widetilde{M}^\ast_n=\frac{2\pi}{\sqrt{n}}
   \sum_{j \in {\mathcal G}(n)} g_{\widehat{\theta}_0}(\lambda_{j,n})  \big( \widetilde{I}_{n}^\ast(\lambda_{j,n}) - \widehat{f}(\lambda_{j,n}) \big),\]
   and  
 $g_{\widehat{\theta}_0}(\lambda)$ the $m$-dimensional vector
\[ g_{\widehat{\theta}_0}(\lambda) = \Big( g_{j,\widehat{\theta}_0}(\lambda) = -\frac{1}{2\pi}\frac{\partial}{\partial \theta_j} f^{-1}_\theta(\lambda)\Big|_{\theta=\widehat{\theta}_0}, \  j=1,2, \ldots, m\Big)^\top.\]
Calculate the pseudo random vector $\widetilde{Z}_n^\ast$ defined by
\[ \widetilde{Z}^\ast_n = \big(\widetilde{V}^\ast_{1,n}\big)^{-1/2}  \widetilde{W}^\ast_n \sqrt{n}(\breve{\theta}^*_n -\widehat{\theta}_0).\]
%   Furthermore,  
%\[ . \]    
\end{enumerate}

\section{Practical Implementation and Numerical Results}
\label{se.numres}

\subsection{Choice of  bootstrap parameters}
To implement our procedure we need to choose two parameters, the bandwidth  involved in obtaining the  spectral density 
estimator $\widehat{f}$ and the  blocksize $b$, used in the convolved part of our procedure. 
Regarding $\widehat{f}$ we use a kernel type estimator obtained by smoothing the periodogram by means of  the Bartlett-Priestley kernel. The  bandwidth used for this estimator 
is obtained using a    frequency domain cross validation procedure; see  Beltr\~ao and Bloomfield (1987).  For the choice of the subsampling parameter $b$, observe first that, as we will see  in the next section, our simulation results seem  not to be  very sensitive with respect to the choice of this parameter, provided $b$  not chosen  too large with respect to $n$. Based on this empirical
observation  we use the following practical rule to select this parameter: $ b =4\cdot n^{0.25}$. This rule delivers a subsampling size which is not   too large  and at the same time also satisfies  the conditions of Assumption 4. However, in order to see the sensitivity of the bootstrap approximations obtained  with  respect to the choice of the subsampling parameter, we present  in Section~\ref{sec.sim}   results  for a large range of values of $b$. The aforementioned rule  for choosing $b$ has been applied  for obtaining the numerical results presented in   Section~\ref{sec.sun} for analyzing the real-life data example.

\subsection{Simulations} \label{sec.sim}
%\subsubsection{Example 1:} \ 
Let  $ {\mathcal F}_\theta$  be   the family containing  the spectral densities of  the first order autoregressive processes  given by  
$$ f_\theta(\lambda) = \frac{\sigma^2}{2\pi}(1+ a^2-2a \cos(\lambda))^{-1},$$
 where $ \theta = (\sigma^2, a) \in (0,+\infty) \times (-1,1)$.  To demonstrate  the advances of the  bootstrap procedure proposed in this paper, we concentrate in the following on the (standard)  
 Whittle estimator of $ a$  given by  
\[ \widehat{a}_n = \sum_{j \in {\mathcal G}(n)} I_n(\lambda_{j,n}) \cos(\lambda_{j,n}) \big/ \sum_{j \in {\mathcal G}(n)} I_n(\lambda_{j,n}). \]
%while the estimator of $ \sigma^2$ is given by  
%\[ \widehat{\sigma}^2_n  = \frac{2\pi}{n}\sum_{j\in {\mathcal G}(n)} I_n(\lambda_{j,n})/\big(1 + \widehat{a}_n^2 - 2 \widehat{a}_n\cos(\lambda_{j,n})\big). \]
Observe that 
$ a_0 = \int_{-\pi}^\pi \cos(\lambda)f_{\theta_0}(\lambda)d\lambda / \int_{-\pi}^\pi f_{\theta_0}(\lambda)d\lambda = \rho(1)$, where $ \rho(1)$ denotes the  first order  autocorrelation of the ``best fitting'' AR(1) process; see Section 2.   
% while 
%$\sigma^2_0 = \int_{-\pi}^\pi f(\lambda)/(1 + a_0 -2a_0\cos(\lambda))d\lambda$.
As  we have seen,   the  distribution of  $ \sqrt{n}(\widehat{a}_n-a_0)$ 
%and $ \sqrt{n}(\widehat{\sigma}^2_n-\sigma^2_0)$,  
depends on the  dependence properties of  the underlying process $ \{X_t;t\in\Z\}$.   In order to demonstrate the finite sample behavior and the capabilities
of the bootstrap procedure proposed to  approximate the distribution of $ \sqrt{n}(\widehat{a}_n-a_0)$ for a variety of situations, 
%under a variety of data gernerating processes, 
we 
consider time series  
$ X_1, X_2, \ldots, X_n$ stemming from the following  three processes:
\begin{enumerate}
\item[] {\bf Model I:} \ $X_t = 0.8 X_{t-1} + \varepsilon_t$, \  and  i.i.d. innovations $ \varepsilon_t\sim {\mathcal N}(0,1)$.
\item[] {\bf Model II:}  \ $ X_t =0.75 X_{t-1} +0.6X_{t-1}\cdot \varepsilon_{t-1} + \varepsilon_t$, \  and  i.i.d. innovations $ \varepsilon_t\sim {\mathcal Laplace}(0,0.1)$.
%\item[] {\bf Model III:}\  $ X_t =  u_t + 0.8 u_{t-1} $,  where   $ u_t =\varepsilon_t\sqrt{1 + 0.25 u^2_{t-1}}$,  and  i.i.d. $ \varepsilon_t\sim {\mathcal N}(0,1)$.
\item[] {\bf Model III:}\  $ X_t = \left\{ \begin{array}{lll}  -0.3 X_{t-1} + \varepsilon_t &  & \mbox{if $X_{t-1}\leq 0$} \\ & & \\ 0.8X_{t-1} + \varepsilon_{t} & & \mbox{if $ X_{t-1} >0$,}
\end{array} \right.$\\
and  i.i.d. innovations $ \varepsilon_t\sim {\mathcal Laplace}(0,0.1)$.
%\item[] {\bf Model IV:}\  $ARCH$ ?????? XXXXXX  
%where $  \varepsilon_t$,   are  i.i.d. innovations $ \varepsilon_t\sim {\mathcal Laplace}(0,0.1)$.
\end{enumerate}
Model I is a Gaussian AR(1) model.  Model II and Model III are nonlinear and   have been considered in Fan and Yao (2005) and Auestad and Tj{\o}stheim (1990), respectively. They have been modified so that they  are driven by Laplace (double exponential) distributed innovations with mean zero and parameter $0.1$.   
Notice that fitting a linear  AR(1) model to  time series stemming from the   nonlinear Models II and III,  resamples a  situation of model misspecification.

Two sample sizes, $n=50$ and $n=1000$ are considered in order to investigate  the small  sample  behavior of the bootstrap procedure proposed as well as its  consistency  behavior  when the length
of the  time series becomes large. In order to see the effects of the convolved subsampling step implemented in Step 4 through generating the pseudo random variables $ M^+_n$,  we also present results for  the bootstrap approximation of $ \sqrt{n}(\widehat{a}_n-a_0)$   using
 the multiplicative periodogram bootstrap only, that  is, the pseudo random variable  $L_{n,MB}^\ast=\sqrt{n}(\widehat{a}^\ast_n - \widehat{a}_0)$, generated in Step 3 of  the bootstrap algorithm of Section 3. 
 $B=1000$ bootstrap  replications have been used  in each run and the $d_1$-distances, between the exact distribution and the two bootstrap approximations,  that is  the multiplicative periodogram bootstrap 
 $  L_{n,MB}^\ast$  and  the hybrid periodogram bootstrap $ L_n^\ast$ as generated in Step 6, have been  calculated.  Notice that for $F$ and $G$ distribution functions,   $d_1=\int_0^1|F^{-1}(u)-G^{-1}(u)|du$.  To estimate the exact distribution of  $ \sqrt{n}(\widehat{a}_n-a_0)$, $R=10,000$ replications have been used. 
Figure~\ref{fig.Whittle} shows  averages of the $d_1$ distances calculated over $500$ repetitions for each of the three different time series models and for each of the two sample sizes considered.

 % depending on  the assumptions one makes on the structure of  $ \{X_t;t\in \Z\}$. 
% HERE REPORT PLOTS \& RESULTS OF  THE SIMULATIONS AND COMMENTS
 \begin{figure}[h]
\begin{center}
\includegraphics[angle=0,height=5.5cm,width=5.8cm]{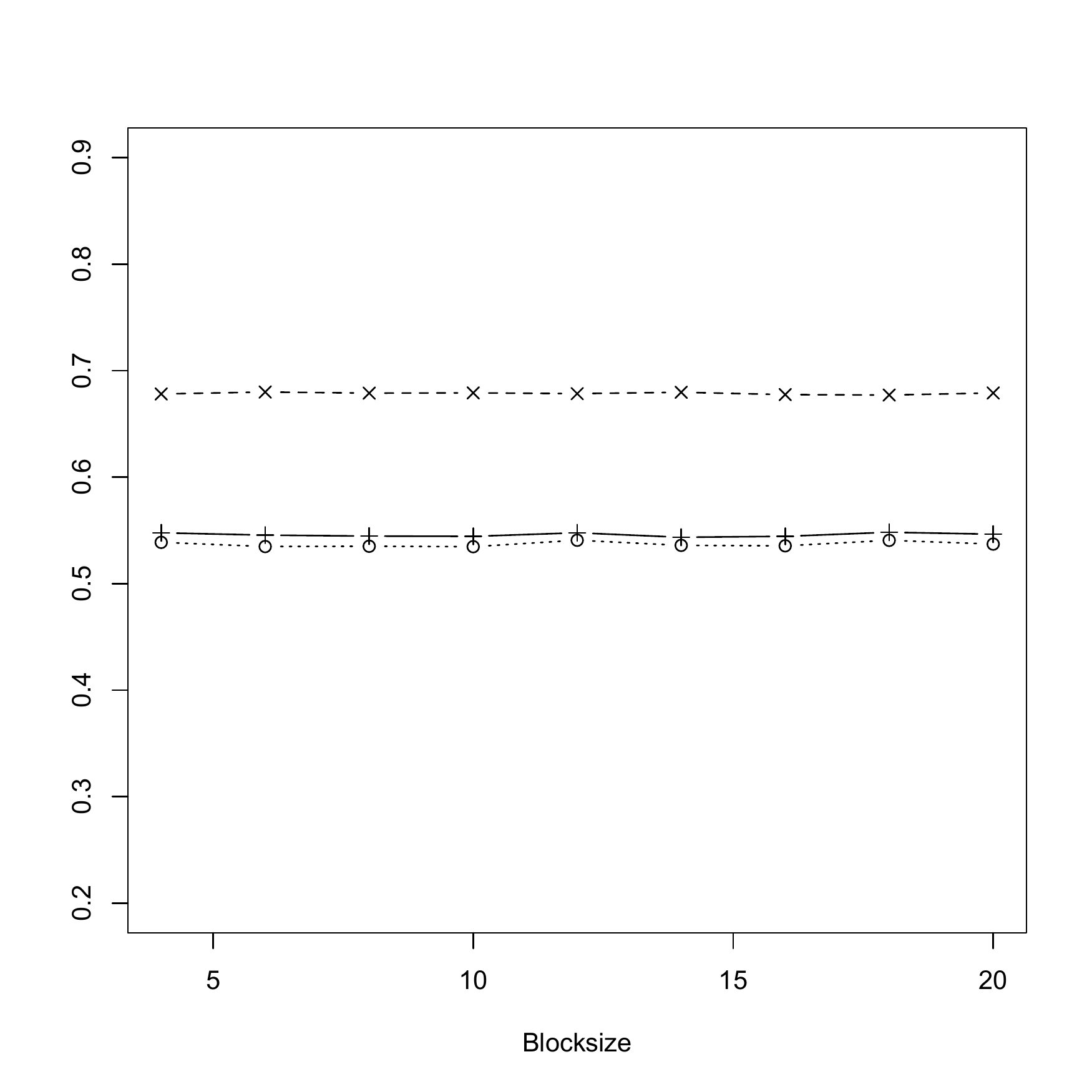}
\includegraphics[angle=0,height=5.5cm,width=5.8cm]{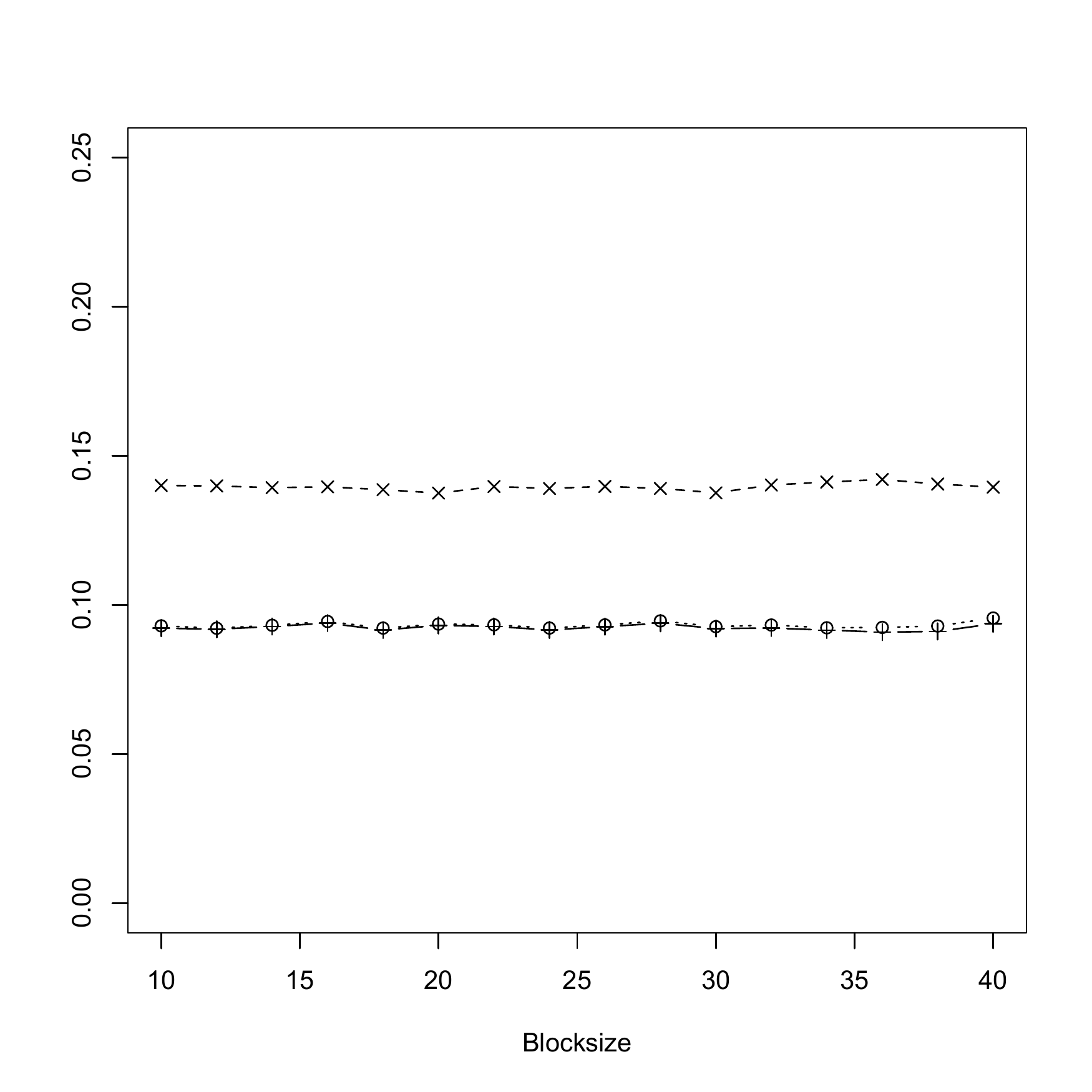}\\
\includegraphics[angle=0,height=5.5cm,width=5.8cm]{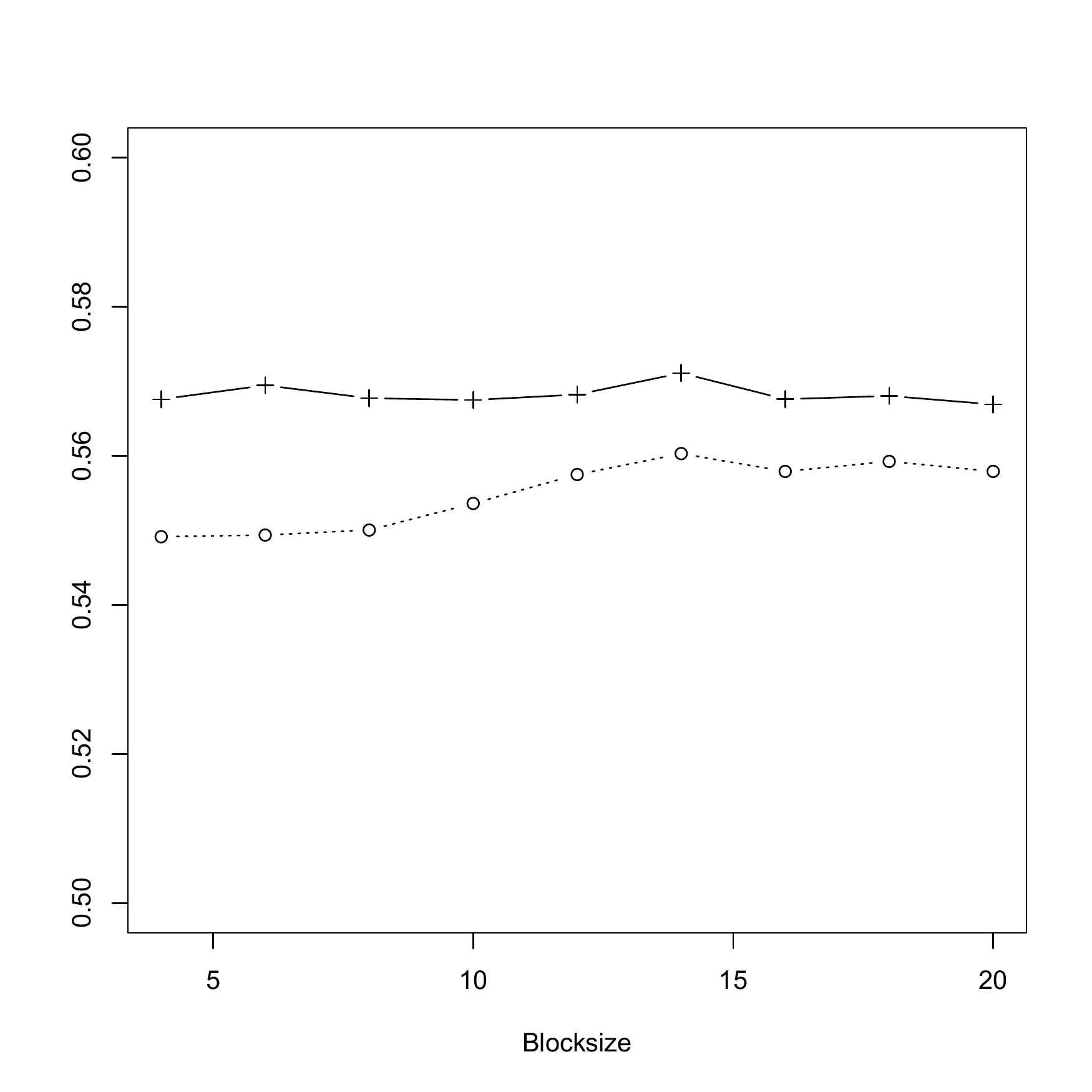}
\includegraphics[angle=0,height=5.5cm,width=5.8cm]{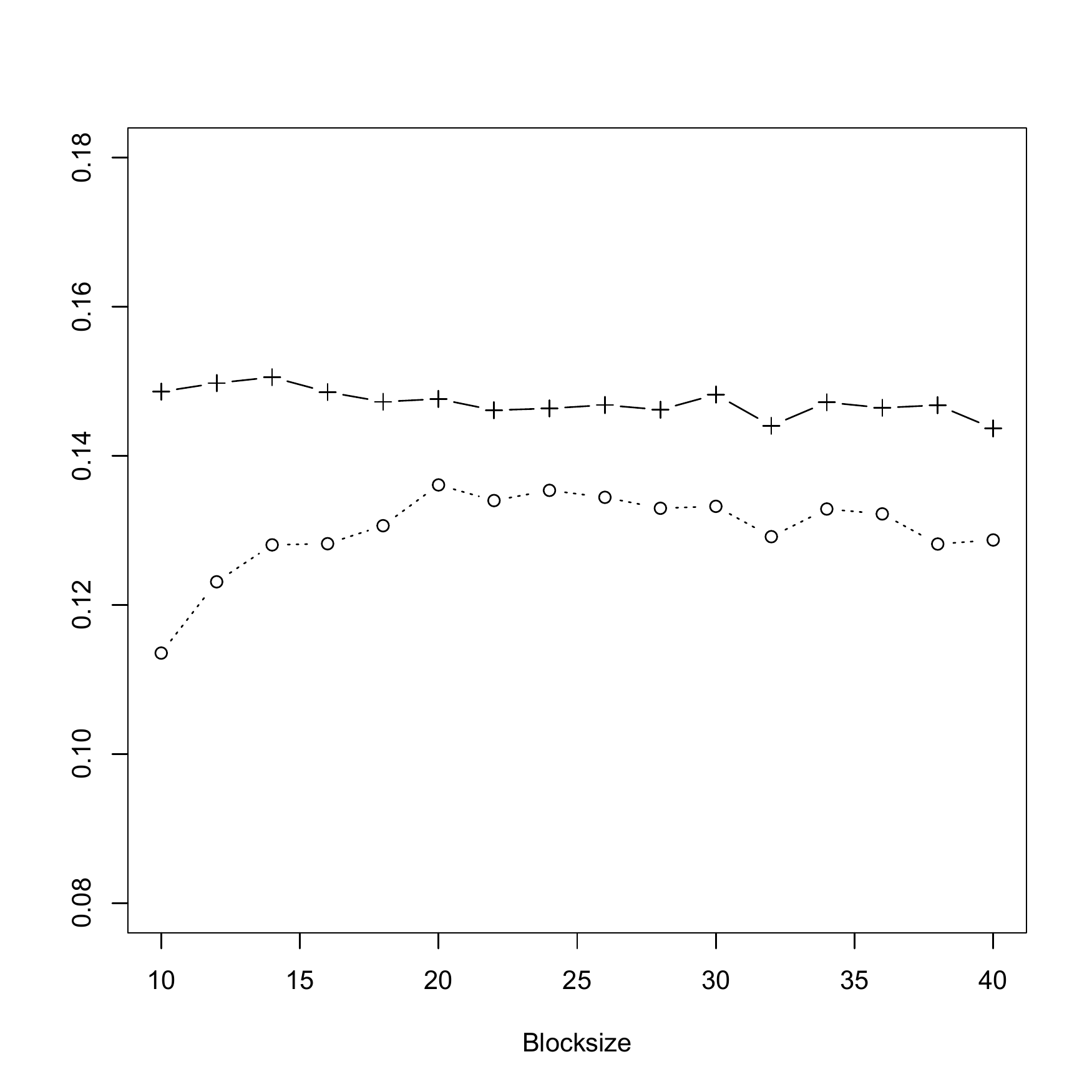}\\
\includegraphics[angle=0,height=5.5cm,width=5.8cm]{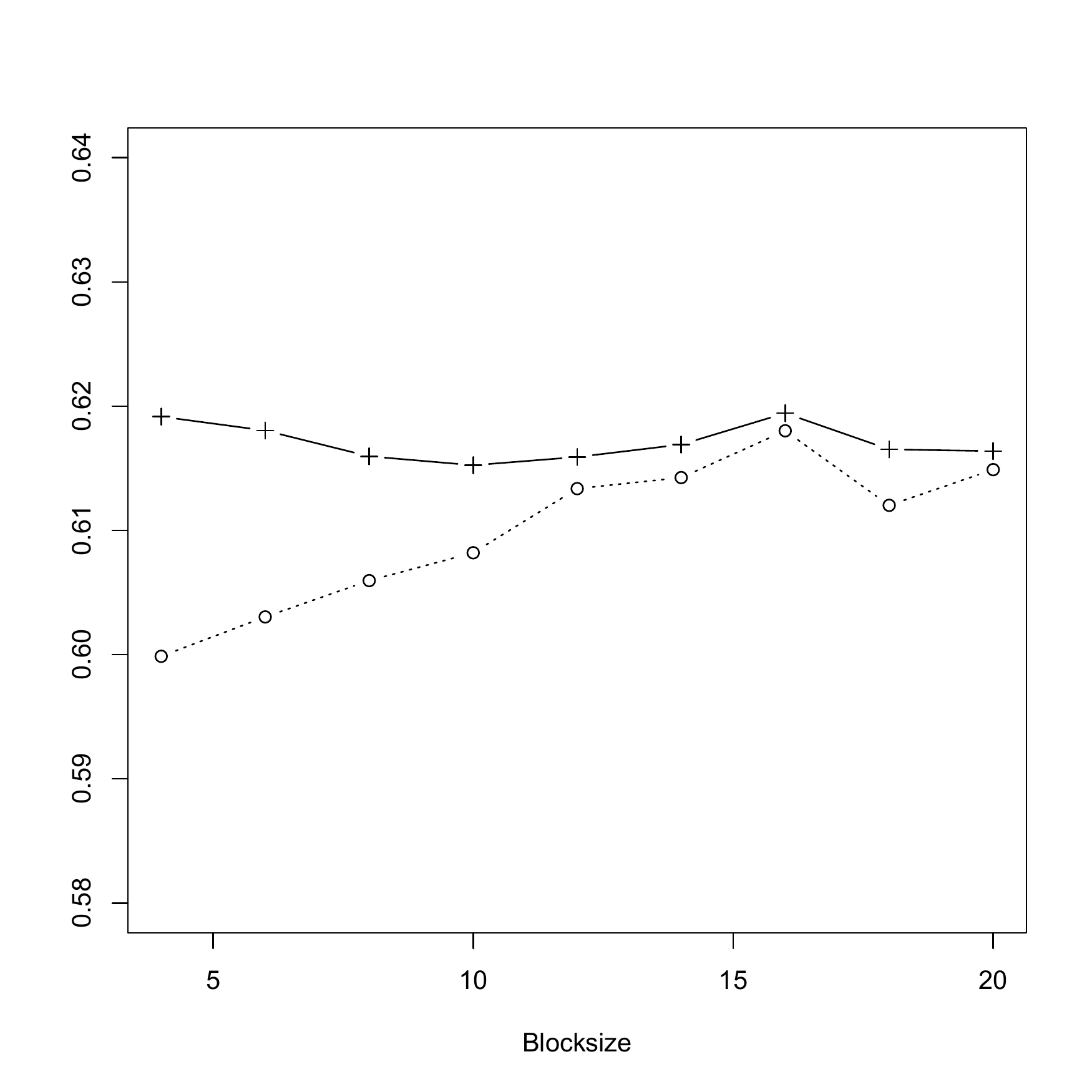}
\includegraphics[angle=0,height=5.5cm,width=5.8cm]{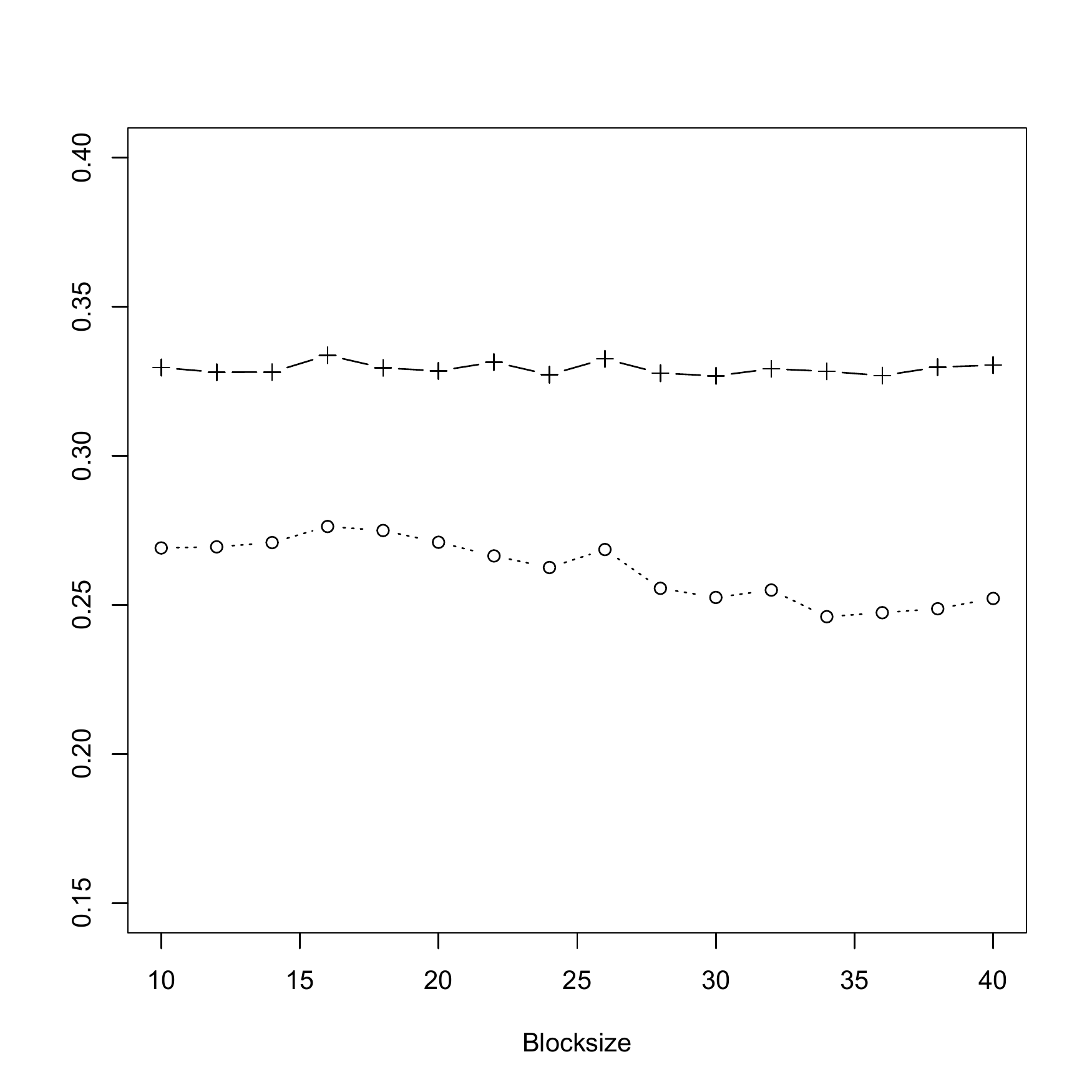}
\end{center}
\caption{Average $d_1$-distances  between the exact and the bootstrap distribution of the Whittle estimator $ \sqrt{n}(\widehat{a}_n-a_0)$ 
for various block sizes $b$. Left column n=50, right column, n=1,000.
First row
Model I,  second row Model II and third  row Model III.   The crosses denote the $d_1$-distance of the multiplicative  periodogram bootstrap estimation  $ L^\ast_{n,MB}$ and 
the circles  of  the hybrid periodogram bootstrap estimation $ L^\ast_n$. The dashed lines with the plus symbol in the first row refer to the average $d_1$-distance of the asymptotic Gaussian approximation.}
\label{fig.Whittle}
\end{figure}
%The simulation results  demonstrate the superiority  of the hybrid bootstrap estimation $ L_n^\ast$ over the 
%multiplicative bootstrap  $ L_{n,MB}^\ast$. 
As it is seen from  Figure~\ref{fig.Whittle}, for the case of the Gaussian AR(1) model and  for both sample sizes considered, the  bootstrap approximations $L_n^\ast$ and $ L_{n,MB}^\ast$ behave 
very similar and the corresponding $d_1$ distances are very close to each other.   Recall that  in this case the distribution of $\sqrt{n}(\widehat{a}_n-a_0) $ only depends on the second order characteristics of the underlying AR(1) process and therefore, the multiplicative bootstrap estimation 
$L_{n,MB}^\ast$ also provides a consistent estimation of the distribution of $ L_n$. As it is also seen, for both sample sizes considered  and for the case of the Gaussian AR(1) model, both  frequency domain  bootstrap procedures outperform the asymptotic Gaussian approximation. This is  due to the skewness of the distribution of $ \sqrt{n}(\widehat{a}-a_0)$ which does not vanish  even for  $n=1000$ observations. In the case of the nonlinear models considered,   
 that is for Model II and  Model III, the behavior of the multiplicative and of the hybrid bootstrap is very different.  Recall that in  these cases, 
  the multiplicative bootstrap fails  to appropriately capture  the fourth order characteristics 
of the underlying nonlinear processes that affect the distribution of $ \sqrt{n}(\widehat{a}_n-a_0)$. This leads to a larger $d_1$ distance of the bootstrap estimator  $ L^\ast_{n,MB}$ compared to   the hybrid  bootstrap estimator  $ L^\ast_n$. The hybrid bootstrap  captures
 these characteristics and performs   
    much better leading to an overall  smaller $d_1$-distance.  For the case $n=50$, this is true for 
   all block sizes $b$  and for all models considered. Only for Model III and for the sample size of $n=50$,  the behavior of the hybrid periodogram bootstrap gets closer to  that of the multiplicative bootstrap  when the  block size $b$ becomes too large.  Finally, for the sample size of $ n=1,000$ observations, the advantages of the hybrid bootstrap procedure are clearly  seen  in  the corresponding exhibits of Figure~\ref{fig.Whittle} for all models. 
   %They seem to be  very stable and  uniform over all block sizes used.

\subsection{Periodicity of Annual Sunspot Data}  \label{sec.sun} \  We consider the  yearly mean total  sunspot numbers from 1700 to 2020 available at   www.sidc.be/silso/datafiles. A  plot of the  corresponding time series consisting of  $n=321$  observations,  is shown  in Figure 2(a). Our aim   is to estimate the main periodicity of this time series and to infer   properties  of the estimator used by applying the frequency domain  bootstrap procedure proposed in this paper. To make things precise, suppose that  the stochastic process generating the observed yearly mean sunspot data
possesses a  spectral density  $f$ and assume that 
a unique frequency $ \lambda_{\max}\in (0,\pi)$ exists such that $ \lambda_{\max} =\arg\min_{\lambda} f(\lambda) $. We are  interested in the main periodicity of the yearly sunspot numbers defined as   the parameter 
 $ P_X=2\pi/\lambda_{\max}$.  
 
 One approach  to estimate this parameter is to use the class of linear  $ AR(p)$ process to get an  estimate of the frequency  $ \lambda_{\max}$. To elaborate, suppose that an AR(p) model is fitted to the time series of sunspot numbers  and that $\widehat{\lambda}_{\max,AR}$ is the (unique) frequency in $(0,\pi)$ defined by  $ \widehat{\lambda}_{\max,AR} =\arg\min_{\lambda} \widehat{f}_{AR}(\lambda) $,
 where $  \widehat{f}_{AR} $ is the spectral density of the estimated AR(p) model. 
 %in order to parametrize the second order structure 
% and consequently the spectral density $f$ of the underlying process; 
 The estimator of $P_X$ obtained following  this approach is  then  defined as  $\widehat{P}_X= 2\pi/\widehat{\lambda}_{\max,AR}$.
 %,  where $   \widehat{\lambda}_{\max} =\arg\min_{\lambda} \widehat{f}_{AR}(\lambda) $ and $ \widehat{f}_{AR}$ denotes the spectral density of the estimated AR(p) model fitted to the time series  of yearly sunspot numbers. 
Observe   that  consistency of the estimator  $ \widehat{P}_X$ only requires  that $\widehat{\lambda}_{\max,AR} \stackrel{P}{\rightarrow} \lambda_{\max}$, as $n\rightarrow\infty$. This can be  achieved if  the  spectral density, say  $f_L$, to which $\widehat{f}_{AR}$  uniformly converges, that is, $ \sup_{\lambda\in[0,\pi]}|\widehat{f}_{AR}(\lambda)-f_L(\lambda)| \stackrel{P}{\rightarrow} 0$, satisfies
    $ \lambda_{\max,L}=\lambda_{\max}$, where $\lambda_{\max,L}= \arg\min_{\lambda} f_{L}(\lambda)$. Hence  consistency of $\widehat{\lambda}_{\max,AR} $ only requires that the limiting spectral density  $f_L$ to which $\widehat{f}_{AR}$ uniformly converges in probability,   has its largest peak at the same frequency as the spectral density $f$ of interest. 
 One way to achieve  this,   is to allow for the order $p$ of the AR model fitted, to increase to infinity at some appropriate rate as $n$ increases to infinity. Under certain conditions  it can then be shown that $ f_L=f$, i.e., $  \sup_{\lambda\in[0,\pi]}|\widehat{f}_{AR}(\lambda)-f(\lambda)| \stackrel{P}{\rightarrow} 0$  and that the corresponding estimator  $ \widehat{P}_{X}=2\pi/\widehat{\lambda}_{\max,AR}$ achieves   the rate $ \widehat{P}_{X} =P_X+O_P(p^{3/2}/n^{1/2})$;
 see  Newton and Pagano (1983).  
 
 However, an alternative way to consistently estimate $ P_X$ is the following. Suppose that there exists a finite order AR(p) model possessing a spectral density  $f_{AR}$ such that $ \lambda_{\max,AR} =\lambda_{\max}$, where $ \lambda_{\max,AR} =\arg\min_{\lambda} f_{AR}(\lambda) $ and $ f_{AR}$ denotes the spectral density of the AR(p) process.    Then $  \sup_{\lambda\in[0,\pi]}|\widehat{f}_{AR}(\lambda)-f_{AR}(\lambda)| \stackrel{P}{\rightarrow} 0$  implies  $ \widehat{P}_{X} = 2\pi/\widehat{\lambda}_{\max,AR} \stackrel{P}{\rightarrow} P_X$  and  this estimator  converges at the parametric  rate $  \widehat{P}_{X} =P_X+O_P(1/n^{1/2})$.  Notice that  consistency of the described approach,     does not rely  on the assumption  that the AR(p) model correctly describes  the  entire stochastic structure of the  process generating  the sunspot data. Not even the entire autocovariance structure of the sunspot time series  has to appropriately be captured by the AR(p) model. What is solely required is that the spectral density  $f_{AR}$  of the AR(p) process  has its main peak at the same frequency as the spectral density $f$ of the  stochastic process generating the sunspot time series.  The  AR(p) model is then solely used  as a vehicle   to  construct  an estimator of  
 %the frequency 
% $ \lambda_{\max}$  which determines 
the main periodicity $P_X$. Moreover, in conjunction with the frequency domain bootstrap,  this approach  also   allows 
for the  investigation of    the sampling properties of the estimator $\widehat{P}_X $  and for quantifying  the uncertainty associated with estimating the parameter $P_X$ of interest.  

For  the yearly sunspot time series, selecting  an AR(p) model  using   Akaike's Information Criterion (AIC),   leads  to an AR(9) model. 
 However, and  as already mentioned,  since we are not interested in parametrizing the entire autocovariance structure of the yearly sunspot numbers but solely in consistently estimating the frequency $\lambda_{\max}$, an  AR(2) model, $ X_t=a_1X_{t-1} + a_2 X_{t-2} + \varepsilon_t$,  also can be used  for this purpose.
 % at which the underlying spectral density reaches its peak.
  Figure 2(b) demonstrates this by showing  the periodogram of the sunspot time series together with the spectral densities  of the fitted  AR(2) and AR(9) models.  Notice that the  periodogram of this time series takes its maximum value at the Fourier frequency $\lambda_{j,n}=0.090625$, which corresponds to a periodicity of $11.034$ years. Fitting  an  AR(2)  model also  leads to the   estimate $\widehat{\lambda}_{\max,AR}=0.090625$, which  corresponds to the  same estimate of the   main periodicity   $\widehat{P}_X=11.034$.
Fitting the   AR(9) model leads to the estimates $   0.09375   $  for  the frequency $\lambda_{\max} $ and   $10.667$ years for  the main periodicity $P_X$.  We, therefore, proceed by  using  the more parsimonious AR(2) model  for our analysis.  In particular, 
we apply the frequency domain bootstrap procedure proposed in this paper to generate replicates of the the estimated parameters $ \widehat{\theta}_n= (\widehat{\sigma}^2, \widehat{a}_1 ,\widehat{a}_2)^\top$ of the AR(2) model.
% say $ \widehat{\theta}_n^\ast = (\sigma^{\ast^2}, \widehat{a}^\ast_1 , \widehat{a}^\ast_2)^\top$.
 Clearly and since $f = f_{AR}$  is  not a reasonable assumption in our context, we are in the setting of model misspecification. Using the bootstrap replicates of the estimated parameters,  we get replicates of the estimated spectral density of the AR(2) model,
 say $\widehat{f}_{AR}^\ast$. 
 Bootstrap replicates 
 of $\widehat{\lambda}_{\max,AR}$ can   then be obtained as $ \widehat{\lambda}_{\max,AR}^\ast =\arg\min_{\lambda} \widehat{f}^\ast_{AR}(\lambda)  $ which lead to  bootstrap replicates 
 $ \widehat{P}^\ast_X=2\pi/\widehat{\lambda}_{\max,AR}^\ast$ of the estimator of the main periodicity $\widehat{P}_{X}=2\pi/\widehat{\lambda}_{\max,AR}$  used. By repeating these steps a large number, say $B$, of times,   bootstrap estimators  of  the distribution of $ \widehat{\lambda}_{\max,AR}$  and, conseqeuntly, of $ \widehat{P}_X$, are obtained.  Figure 2(c) presents a histogram of  $B=1,000$ bootstrap replicates of $\widehat{\lambda}_{\max,AR}^\ast$  and Figure 2(d)   of  the corresponding estimates $ \widehat{P}^\ast_X=2\pi/\widehat{\lambda}_{\max,AR}^\ast$ obtained by using   a grid of $500$ equidistant frequencies in the interval $(0,\pi)$.  The corresponding  $95\%$ confidence interval for the main periodicity of the sunspot time series, based on   bootstrap percentages, is then given by $[9.90, 12.98]$.
 
  \begin{figure}[h]
\begin{center}
\includegraphics[angle=0,height=16cm,width=16cm]{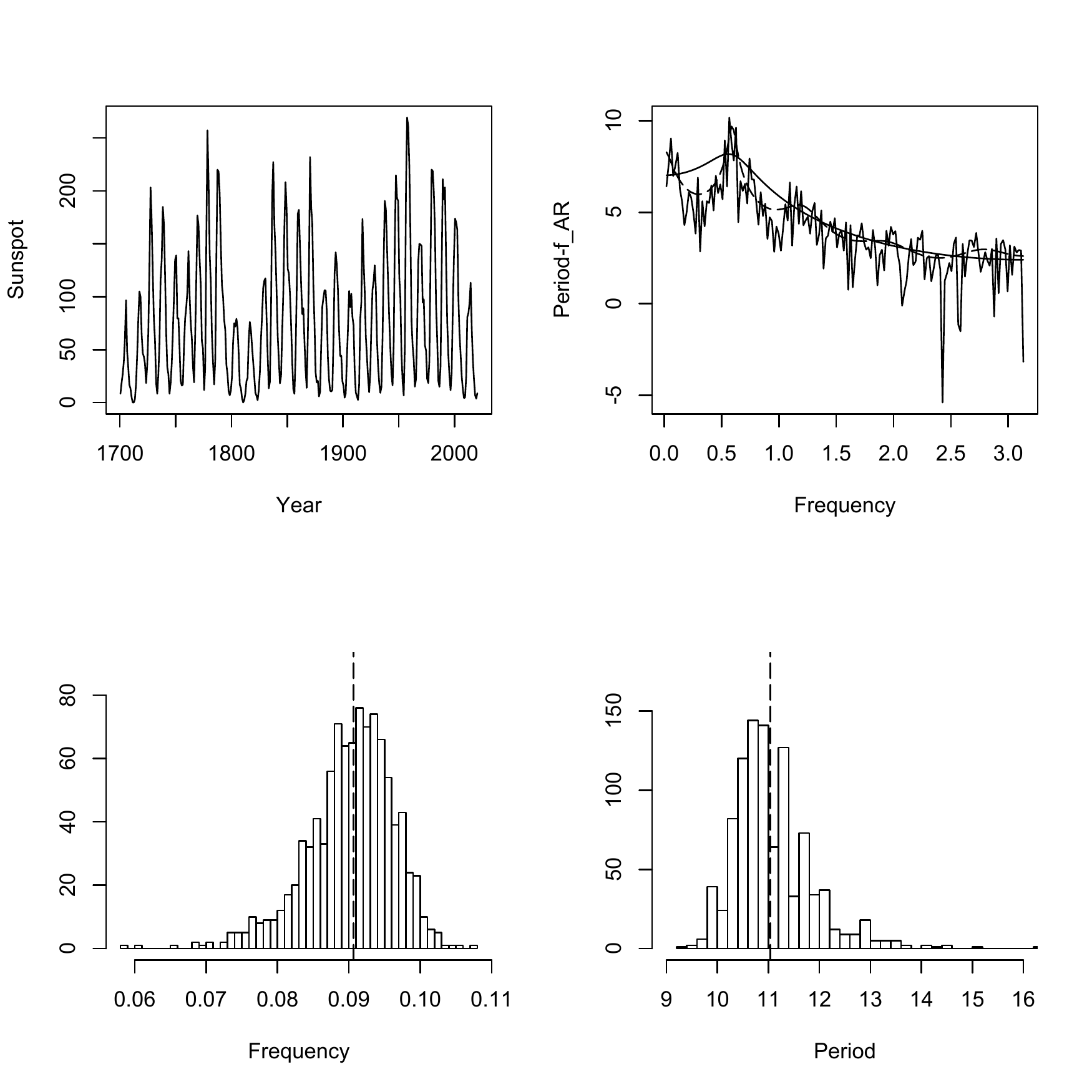}\\
% \vspace*{-8cm}
%\includegraphics[angle=0,height=16cm,width=16cm]{Hplot.pdf}
\end{center}
%\vspace*{-8cm}
\caption{Reading clockwise from top to bottom: (a) Time series of yearly mean sunspot numbers. (b) Periodogram of the time series with estimated spectral densities of the AR(2) model (solid line) and of the AR(9) model (dashed line), log-scale. (c) Histogram of $B=1000$ replications of  $ \widehat{\lambda}^\ast_{\max,AR} $ and  (d) histogram of the corresponding  replications of the estimated main periodicity$ \widehat{P}^\ast_X$, both   using  the AR(2) model. The estimated values $ \widehat{\lambda}_{\max,AR}$ and $ \widehat{P}_{X}$ are indicated in (c) and (d) by vertical dashed lines.}
\label{fig.Sunspot}
\end{figure}

\section{Auxiliary Results and Proofs}  \label{sec.7}

To simplify notation we write  $ D_n^{(j)}(\mathring \theta,I_n^\ast)$, $\mathring \theta\in\Theta$,  for the $j$th derivative of $ D_n(\theta, I^\ast_n)$ with respect to $\theta$ evaluated at $ \theta=\mathring\theta$.  We also write $ \|\cdot \|$ for the Euclidean norm in $\Re^m$ and $ \|A\|_F$ for the Frobenius norm of a matrix $ A \in \Re^{m\times m}$.  Furthermore, at different places we will use  the expansion
\begin{equation} \label{eq.ML-Taylor}
D^{(1)}_n(\theta,I^\ast_n) = D^{(1)}_n(\widehat{\theta}_0,I^\ast_n) + D^{(2)}_n(\widehat{\theta}_0,I^\ast_n)(\theta-\widehat{\theta}_0)  +R_n(\theta)(\theta-\widehat{\theta}_0),
\end{equation}
where $ R_n(\theta)=D^{(2)}_n(\theta^+,I^\ast_n)-D^{(2)}_n(\widehat{\theta}_0,I^\ast_n)$ for  some $\theta^+\in \Theta$ such that $ \|\theta^+-\widehat{\theta}_0\|\leq \|\theta-\widehat{\theta}_0\|$.

We  first establish the following two   useful lemmas.

\begin{lemma} \label{le.Appendix1} If  Assumption 2 and Assumption 3 are satisfied, then the following assertions hold true. 
\begin{enumerate}
\item[(i)]\  \   $ \widehat{\theta}_0 \stackrel{P}{\rightarrow} \theta_0$.
\item[(ii)] \ $ \big\| \widehat{\theta}_n^\ast -\widehat{\theta}_0\big\| \stackrel{P}{\rightarrow} 0$, in probability.
\end{enumerate}
\end{lemma}
\begin{proof}
Consider  (i). Since  $\widehat{\theta}_0 \in \Theta$ and $\Theta \subset  {\mathbb R}^m$ is compact, $\{\widehat{\theta}_0,n\in\N\}$ is a bounded sequence. 
By the continuity of $D(\cdot,f)$ as a function on $\theta \in \Theta$, it suffices to show that 
$$ D(\widehat{\theta}_0, f) \stackrel{P}{\rightarrow} D(\theta_0,f).$$
It yields
\begin{align*}
|D(\widehat{\theta}_0,f) -D(\theta_0,f)| & \leq |D(\widehat{\theta}_0,f) -D_n(\widehat{\theta}_0,f)| + |D_n(\widehat{\theta}_0,f) -D_n(\widehat{\theta}_0,\widehat{f})|\\
& \ \ \ \ + |D_n(\widehat{\theta}_0,\widehat{f}) -D(\widetilde{\theta}_0,\widehat{f})| + |D(\widetilde{\theta}_0,\widehat{f}) -D(\theta_0,f)| \\
&= \sum_{j=1}^4 D_{j,n},
\end{align*}
where $ \widetilde{\theta}_0 =\mbox{argmin}_{\theta} D(\theta,\widehat{f})$ and with an obvious notation for $D_{j,n} $, $j=1,\ldots, 4$.
We show that $  D_{j,n} \stackrel{P}{\rightarrow} 0$, as $ n \rightarrow\infty$, for $ j=1, \ldots,4$.
We have
\begin{align*}
 D_{1,n} & \leq  \sup_{\theta\in \Theta}\big|\frac{1}{2\pi}\int_{-\pi}^\pi \log f_\theta(\lambda)d\lambda - \frac{1}{n}\sum_{j\in{\mathcal G}(n)} \log f_\theta(\lambda_{j,n}) \big|\\
& \ \  +\sup_{\theta\in\Theta} \big|\frac{1}{2\pi} \int_{-\pi}^\pi \frac{f(\lambda)}{f_\theta(\lambda)}d\lambda -   \frac{1}{n}
\sum_{j\in{\mathcal G}(n)}  \frac{f(\lambda_{j,n})}{f_\theta(\lambda_{j,n})}\big|  = O(n^{-1}), 
\end{align*}
where the last equality follows by the differentiability of $ 1/f_{\theta}(\lambda) $ with respect to $\theta$, the boundedness properties of  $f_\theta\in {\mathcal F}_\theta$  and  the fact that $ \sup_{\theta\in\Theta}\big|\partial/\partial \theta f^{-1}_\theta(\lambda) \big| $ is bounded uniformly in $\lambda$.
Using $ 1/f_\theta(\lambda)  \leq 1/\delta$  and Assumption 3, we get 
\begin{align*}
 D_{2,n} & \leq \sup_{\lambda\in [-\pi,\pi]}\big| \widehat{f}(\lambda) -f(\lambda)\big| \sup_{\theta\in\Theta}\frac{1}{n}\sum_{j\in{\mathcal G}(n)} \frac{1}{f_\theta(\lambda_{j,n})}  \\
& = O(1)\sup_{\lambda\in [-\pi,\pi]}\big| \widehat{f}(\lambda) -f(\lambda)\big| \stackrel{P}{\rightarrow} 0.
\end{align*}
To establish   $ D_{3,n}\stackrel{P}{\rightarrow} 0$, it suffices to show  that  $ \sup_{\theta\in\Theta}|D_n(\theta,\widehat{f}) -D(\theta,\widehat{f})|\stackrel{P}{\rightarrow} 0$. For this we have 
\begin{align*}
 \sup_{\theta\in\Theta}|D_n(\theta,\widehat{f}) -D(\theta,\widehat{f})| & \leq  
  \sup_{\theta\in \Theta}\big| \frac{1}{n}\sum_{j\in{\mathcal G}(n)} \log f_\theta(\lambda_{j,n}) - \frac{1}{2\pi}\int_{-\pi}^\pi \log f_\theta(\lambda)d\lambda \big|\\
& \ \  +\sup_{\theta\in\Theta} \big|   
\frac{1}{n}\sum_{j\in{\mathcal G}(n)}  \frac{\widehat f(\lambda_{j,n})}{f_\theta(\lambda_{j,n})} - \frac{1}{2\pi} \int_{-\pi}^\pi \frac{\widehat f(\lambda)}{f_\theta(\lambda)}d\lambda\big|\\
& \leq O(n^{-1}) + \sup_{\theta\in\Theta} \big|\frac{1}{n}\sum_{j\in{\mathcal G}(n)} \frac{1}{f_\theta(\lambda)}(\widehat{f}(\lambda_{j,n}) - f(\lambda_{j,n}) )  \big| \\
& \ \  + \sup_{\theta\in\Theta} \big|\frac{1}{n}\sum_{j\in{\mathcal G}(n)} \frac{f(\lambda_{j,n})}{f_\theta(\lambda_{j,n})} -\frac{1}{2\pi} \int_{-\pi}^\pi \frac{f(\lambda)}{f_\theta(\lambda)}d\lambda   \big| \\
& \ \  +  \sup_{\theta\in\Theta} \big|\frac{1}{2\pi} \int_{-\pi}^\pi \frac{1}{f_\theta(\lambda)} ( f(\lambda)- \widehat f(\lambda))d\lambda   \big|\\
& = O(n^{-1}) + O(1)\sup_{\lambda\in [-\pi,\pi]}\big| \widehat{f}(\lambda) -f(\lambda)\big|,
%+ O(n^{-1})  + O(1)\sup_{\lambda\in [-\pi,\pi]}\big| \widehat{f}(\lambda) -f(\lambda)\big|.
\end{align*}
where the last equality follows because   the second and the last term of the last bound above  is $ O(1)\sup_{\lambda\in [-\pi,\pi]}\big| \widehat{f}(\lambda) -f(\lambda)\big|$ and the  third term is $ O(n^{-1})$.  Finally, 
 $ D_{4,n}\stackrel{P}{\rightarrow} 0$ follows from $ \sup_{\theta\in\Theta}|D(\theta,\widehat{f}) -D(\theta,f)|\stackrel{P}{\rightarrow} 0$, which holds true  since  
\begin{align*}
 \sup_{\theta\in\Theta}|D(\theta,\widehat{f}) -D(\theta,f)| &  \leq \sup_{\theta\in\Theta} \big|\frac{1}{2\pi} \int_{-\pi}^\pi \frac{1}{f_\theta(\lambda)} ( f(\lambda)- \widehat f(\lambda))d\lambda   \big|\\
 & \leq  O(1)\sup_{\lambda\in [-\pi,\pi]}\big| \widehat{f}(\lambda) -f(\lambda)\big|.
\end{align*}

Consider  (ii).   
Recall that by Assumption 3, the matrix $ W$ is nonsingular. By Lemma 4.2 of Lahiri (2003),  $ D^{(2)}_n(\widehat{\theta}_0,I_n^\ast) $ is nonsingular, if  for $\delta>0$, 
\[ \|D^{(2)}_n(\widehat{\theta}_0,I_n^\ast)-W\|_F \leq \delta/\|W^{-1}\|_F,\]
on a set with  probability arbitrarily close to one for $n$ large enough. This  holds true since $ \|D^{(2)}_n(\widehat{\theta}_0,I_n^\ast) -W\|_F \stackrel{P}{\rightarrow} 0$, in probability, see Lemma 7.2(i).  Furthermore,  by the same Lemma 4.2, we have, on the same set, that, with  probability arbitrarily close to one,  
\begin{equation} \label{eq.InvertD2}
 \|  (D^{(2)}_n(\widehat{\theta}_0,I_n^\ast))^{-1}\|_F \leq \|W^{-1}\|_F/(1-\delta) = 2\|W^{-1}\|_F,
 \end{equation}
for $ \delta=1/2$.
% Let $ A_n$ be the set on which $ \|D^{(2)}_n(\widehat{\theta}_0,I_n^\ast)-W\|_F \leq \delta/\|W^{-1}\|_F$ holds true and notice that $ P(A_n)\rightarrow1$ as $ n \rightarrow\infty$.
Recall  equation (\ref{eq.ML-Taylor}) and define on the set on which (\ref{eq.InvertD2})  holds true  the function 
%$\widetilde{R}_n(\theta) = R_n(\theta)(\theta-\widehat{\theta}_0)$ and define the function 
\begin{equation}\label{eq.w-function}
w(\widehat{\theta}_0-\theta) = \big(D_n^{(2)}(\widehat{\theta}_0,I_n^\ast)\big)^{-1}\big[D_n^{(1)}(\widehat{\theta}_0,I^\ast_n) + \widetilde{R}_n(\theta) \big],  \ \ \ \theta\in {\mathcal B}(\widehat{\theta}_0,\delta),
\end{equation}
where $ {\mathcal B}(\widehat{\theta}_0,\delta) =\{x\in \Re^m| \|x\|\leq \delta\}$ and
\[ \widetilde{R}_n(\theta) = \big( D_n^{(2)}(\theta^+_n,I_n^\ast) - D_n^{(2)}(\widehat{\theta}_0,I_n^\ast)\big) (\theta-\widehat{\theta}_0) ,\]
with $ \|\widehat{\theta}_0- \theta^+_n\| \leq \|\widehat{\theta}_0-\theta\|$. Using the expansion
\[ D_n^{(1)}(\theta, I^\ast_n) = D_n^{(1)}(\widehat{\theta}_0,I^\ast_n) + D_n^{(2)}(\widehat{\theta}_0,I^\ast_n) (\theta-\widehat{\theta}_0) + \widetilde{R}_n(\theta),\]
we can also express $w(\cdot)$ as  
\[ w(\widehat{\theta}_0-\theta) = \big(D_n^{(2)}(\widehat{\theta}_0,I^\ast_n) \big)^{-1} \big[ D_n^{(1)}(\theta,I^\ast_n) - D_n^{(2)}(\widehat{\theta}_0, I^\ast_n)(\theta-\widehat{\theta}_0) \big].\]
From this and because,  by assumption, $ \widehat{\theta}_n^\ast$ is  the unique solution of $ D_n^{(1)}(\theta, I^\ast_n)=0$, we get that  for $\theta=\widehat{\theta}_n^\ast$  it holds true that $  w(\widehat{\theta}_0-\widehat{\theta}_n^\ast) =\widehat{\theta}_0-\widehat{\theta}_n^\ast$ and this is the unique solution of 
$w(\widehat{\theta}_0-\theta) =\widehat{\theta}_0-\theta$ .
% where $ \widehat{\theta}_n^\ast$ is, by assumption, the unique solution of $ D_n^{(1)}(\theta, I^\ast_n)=0$.

We next show that for $ n$ large enough and with probability arbitrarily close to one, a  constant $ C>0$ exists such that $ \|w(\widehat{\theta}_0-\theta) \| \leq C\log(n)/\sqrt{n}$ 
 if $ \|\widehat{\theta}_0-\theta\|\leq C\log(n)/\sqrt{n}$. Toward this goal,  we get using (\ref{eq.w-function})  and the fact that $ D_n^{(1)}(\widehat{\theta}_0, \widehat{f}) =0$,  the bound
\begin{align} \label{eq.bound-w}
\|w(\widehat{\theta}_0-\theta) \|  \leq & \|\big(D_n^{(2)}(\widehat{\theta}_0,I_n^\ast)\big)^{-1}\|_F\Big( \|D_n^{(1)}(\widehat{\theta}_0,I^\ast_n)  - D_n^{(1)}(\widehat{\theta}_0,\widehat{f}) 
 \| + \|\widetilde{R}_n(\theta) \big)\|\Big).
\end{align}
Recall  the bound  $\|  (D^{(2)}_n(\widehat{\theta}_0,I_n^\ast))^{-1}\|_F \leq 2\|W^{-1}\|_F$. Furthermore, 
for the first term in parentheses on the right hand side of (\ref{eq.bound-w}), we have  
\[ \|D_n^{(1)}(\widehat{\theta}_0,I^\ast_n)  - D_n^{(1)}(\widehat{\theta}_0,\widehat{f}) 
 \|  = \|Y^\ast_n\|,\]
 where  
 $$ Y^\ast_n= \frac{1}{n}\sum_{j\in{\mathcal G}(n)} \frac{\partial}{\partial\theta} \frac{1}{f_\theta(\lambda_{j,n})}\Big|_{\theta=\widehat{\theta}_0}\widehat{f}(\lambda_{j,n})(U_j^\ast-1).$$
It yields,
\begin{align*}
n{\rm E}^\ast\|Y^\ast_n\|^2 = & \frac{2}{n}\sum_{j\in{\mathcal G}(n)}\Big( \frac{\partial}{\partial\theta} \frac{1}{f_\theta(\lambda_{j,n})}\Big|_{\theta=\widehat{\theta}_0}\Big)^\top\Big(\frac{\partial}{\partial\theta} \frac{1}{f_\theta(\lambda_{j,n})}\Big|_{\theta=\widehat{\theta}_0} \Big)\widehat{f}(\lambda_{j,n})^2\\
& \stackrel{P}{\rightarrow} \frac{1}{\pi} \int_{-\pi}^\pi \Big( \frac{\partial}{\partial\theta} \frac{1}{f_\theta(\lambda)}\Big|_{\theta=\theta_0}\Big)^\top\Big(\frac{\partial}{\partial\theta} \frac{1}{f_\theta(\lambda)}\Big|_{\theta=\theta_0} \Big)f(\lambda)^2, 
\end{align*}
by Assumption 3, the continuity of the derivative and  Lemma 7.1(i). This implies by Markov's inequality, that,
\begin{align*}
{\rm P}^\ast\big(\|D_n^{(1)}(\widehat{\theta}_0,I^\ast_n)  - D_n^{(1)}(\widehat{\theta}_0,\widehat{f}) 
 \|  \geq C\log(n)/\sqrt{n}\big)& \leq \frac{n{\rm E}\|Y_n^\ast\|^2}{C^2\log^2(n)}  =O_P(1/\log^2(n)).
\end{align*}
 For the term $ \widetilde{R}_n(\theta)$ in (\ref{eq.bound-w}) we use the bound 
\[ \|\widetilde{R}_n(\theta)\| \leq \big( \|M_{1,n}\|_F + \|M_{2,n}\|_F\big) \|\theta-\widehat{\theta}_0\|,\]
where 
\[ M_{1,n}=\frac{1}{n} \sum_{j\in{\mathcal G}(n)} \Big(\frac{\partial^2}{\partial\theta \partial\theta^{\top}}\log f_\theta (\lambda_{j,n})\Big|_{\theta=\theta_n^+} - 
 \frac{\partial^2}{\partial\theta \partial\theta^{\top}}\log f_\theta (\lambda_{j,n})\Big|_{\theta=\widehat{\theta}_0}\Big)\]
and
\[ M_{2,n}=\frac{1}{n} \sum_{j\in{\mathcal G}(n)} \Big(\frac{\partial^2}{\partial\theta \partial\theta^{\top}}\frac{1}{f_\theta (\lambda_{j,n})}\Big|_{\theta=\theta_n^+} - 
 \frac{\partial^2}{\partial\theta \partial\theta^{\top}}\frac{1}{f_\theta (\lambda_{j,n})}\Big|_{\theta=\widehat{\theta}_0}\Big)I^\ast_n(\lambda_{j,n}).\]
 By the   Lipschitz continuity of the second order derivatives, following from Assumption 3, 
 %on the compact set $ \Theta \times [-\pi,\pi]$  
 and since $ \|\theta_n^+-\widehat{\theta}_0\| \leq \|\theta- \widehat{\theta}_0\|$,  we get  $ \|\widetilde{R}_n(\theta) \| \leq  C\|\widehat{\theta}_0 -\theta\|$.
 Hence   if  $ \|\widehat{\theta}_0-\theta\|\leq C\log(n)/\sqrt{n}$ then,   for $n$ large enough and with probability arbitrarily close to one,  we have that, $ \|w(\widehat{\theta}_0-\theta) \| \leq C\log(n)/\sqrt{n}$. Consider  next the function $ g : {\mathcal B}(0,1)\rightarrow {\mathcal B}(0,1)$ defined  as 
 \[ g(x) = \frac{\sqrt{n}}{C\log(n)} w\big(\frac{C\log(n)}{\sqrt{n}}\cdot x\big).\]
 Notice that $ g$ is continuous and that  because for $ x\in{\mathcal B}(0,1)$, $ \| C\log(n) x/\sqrt{n}\| \leq C\log(n)/\sqrt{n}$, we have,
 \[ \|g(x)\|  = \frac{\sqrt{n}}{C\log(n)} \|w(\frac{C\log(n)}{\sqrt{n}} x) \| \leq \frac{\sqrt{n}}{C\log(n)}\cdot \frac{C \log(n)}{\sqrt{n}}    =1.\]
  By Bronwer's fixed point Theorem, see Lahiri (2003), Proposition 4.1, there exists  $x_0\in {\mathcal B}(0,1)$  such that $g(x_0)=x_0$, that is,
 \[ w\big(\frac{C\log(n)}{\sqrt{n}}\cdot x_0\big) =  \frac{{C\log(n)}}{\sqrt{n}} x_0.\]
 Since  $ \widehat{\theta}_0 -\widehat{\theta}^\ast_n$ is the unique  solution of $ w(\widehat{\theta}_0-\theta) = \widehat{\theta}_0-\theta$, we have  that 
 $  \widehat{\theta}_0 -\widehat{\theta}^\ast_n =x_0 C\log(n)/ \sqrt{n} $, that is, 
 \[ \| \widehat{\theta}_0 -\widehat{\theta}^\ast_n\| \leq \frac{C\log(n)}{\sqrt{n}}\|x_0\| \leq \frac{ C\log(n)}{\sqrt{n}}.\]
Hence for $n$ large enough and   with probability arbitrarily close to one, we have 
\[ \| \widehat{\theta}_0 -\widehat{\theta}^\ast_n\| = O_P\Big( \frac{\log(n)}{\sqrt{n}}\Big),\]
which converges to zero, as $n\rightarrow \infty$.
\end{proof}

\begin{lemma} \label{le.Appendix2}  Suppose that Assumption 1 to Assumption 4 are satisfied.  Then, as $ n \rightarrow\infty$,
\begin{enumerate}
\item[(i)] \ $ W^\ast_n \stackrel{P}{\rightarrow} W$,
\item[(ii)] \ $V_{1,n}^\ast \stackrel{P}{\rightarrow} V_1$, 
\item[(iii)] \  $V^+_{2,n} \stackrel{P}{\rightarrow} V_2$.
\end{enumerate}
\end{lemma}
\begin{proof} 
Notice first that  
\begin{equation} \label{eq.lem2-1}
\sup_{\lambda \in[\pi,\pi]}\big\| g_{\widehat{\theta}_0}(\lambda) -  g_{\theta_0}(\lambda) \big\| \stackrel{P}{\rightarrow} 0.
\end{equation}
and 
\begin{equation} \label{eq.lem2-2}
\sup_{\lambda\in[\pi,\pi]}\Big\| \frac{\partial^2}{\partial\theta\partial\theta^\top}f^{-1}_\theta(\lambda)\Big|_{\theta=\widehat{\theta}_0} - 
\frac{\partial^2}{\partial\theta\partial\theta^\top}f^{-1}_\theta(\lambda)\Big|_{\theta=\theta_0}   \Big\|_F \stackrel{P}{\rightarrow} 0.
\end{equation}
To see why the  above assertions hold true, observe first that   by  Lemma~\ref{le.Appendix1}(i) we have   for $n$ large enough, $ P(\widehat{\theta}_0\in B(\theta_0,\epsilon))\geq 1-\epsilon$,  where $ B(\theta_0, \epsilon)=\{\theta: \|\theta-\theta_0\| \leq \epsilon\} \subset \Theta$.
(\ref{eq.lem2-1}) and  (\ref{eq.lem2-2}) follow then because the functions $g_\theta (\lambda)$ and  
$\partial^2\big/(\partial\theta\partial\theta^\top)f^{-1}_\theta(\lambda) $ are  uniformly continuous on the compact set $ [-\pi,\pi]\times B (\theta_0, \epsilon)$.
%, it is also uniformly continuous on the same set and, therefore, (\ref{eq.lem2-2}) is true since XXX. 

Consider (i). We have
\begin{align*}
W^\ast_n & = \frac{1}{n}\sum_{j\in{\mathcal G}(n)}  \frac{\partial^2}{\partial\theta\partial\theta^\top}\Big( \log f_\theta(\lambda_{j,n}) + I^\ast_n(\lambda_{j,n}) f^{-1}_\theta(\lambda)\Big)\Big|_{\theta=\theta_0}  \\
& \ \ +\frac{1}{n}\sum_{j\in{\mathcal G}(n)} \Big(  \frac{\partial^2}{\partial\theta\partial\theta^\top} \log f_\theta(\lambda_{j,n})\Big|_{\theta=\widehat{\theta}_0} -  \frac{\partial^2}{\partial\theta\partial\theta^\top} \log f_\theta(\lambda_{j,n})\Big|_{\theta=\theta_0}\Big)\\
& \  \ + \frac{1}{n}\sum_{j\in{\mathcal G}(n)} \Big( \frac{\partial^2}{\partial\theta\partial\theta^\top} f^{-1}_\theta(\lambda_{j,n})\Big|_{\theta=\widehat{\theta}_0}
- \frac{\partial^2}{\partial\theta\partial\theta^\top} f^{-1}_\theta(\lambda_{j,n})\Big|_{\theta=\theta_0}\Big) I_n^\ast(\lambda_{j,n})\\
&=  \frac{1}{n}\sum_{j\in{\mathcal G}(n)}  \frac{\partial^2}{\partial\theta\partial\theta^\top}\log f_\theta(\lambda_{j,n})\Big|_{\theta=\theta_0} 
+ \frac{1}{n}\sum_{j\in{\mathcal G}(n)}
  \frac{\partial^2}{\partial\theta\partial\theta^\top}f^{-1}_\theta(\lambda)\Big|_{\theta=\theta_0}  I^\ast_n(\lambda_{j,n})  + o_P(1) \\
&\stackrel{P}{\rightarrow}   \frac{\partial^2}{\partial\theta\partial\theta^\top}D(\theta_0,f).
\end{align*}
The last equality follows because by  (\ref{eq.lem2-2}),
\begin{align*}
\Big\|\frac{1}{n}\sum_{j\in{\mathcal G}(n)} &\Big(   \frac{\partial^2}{\partial\theta\partial\theta^\top} f^{-1}_\theta(\lambda_{j,n})\Big|_{\theta=\widehat{\theta}_0}
- \frac{\partial^2}{\partial\theta\partial\theta^\top} f^{-1}_\theta(\lambda_{j,n})\Big|_{\theta=\theta_0}\Big) I_n^\ast(\lambda_{j,n})\Big\|_F\\
& \leq \sup_{\lambda\in[-\pi,\pi]} \Big\| \frac{\partial^2}{\partial\theta\partial\theta^\top}f^{-1}_\theta(\lambda)\big|_{\theta=\widehat{\theta}_0} - 
\frac{\partial^2}{\partial\theta\partial\theta^\top}f^{-1}_\theta(\lambda)\big|_{\theta=\theta_0}   \Big\|_F\times  \frac{1}{n}\sum_{j\in{\mathcal G}(n)}I^\ast_n(\lambda_{j,n})\\
& = o_P(1)\cdot O_P(1).
\end{align*}
Consider (ii).  Recall that $V_{1,n}^\ast ={\rm Var}^\ast(M^\ast_n)$  and   that
$ {\rm Cov}(I^\ast_n(\lambda_{j,n}),I^\ast_n(\lambda_{k,n}) )= {\bf 1}_{\{|j|=|k|\}} \widehat{f}(\lambda_{j,n})^2$. Hence 
% ( by the independence of the pseudo periodogram ordinates $ I^\ast_n(\lambda_{j,n})$ for $j=1,2, \ldots, N$, we have using 
\begin{align*}
{\rm Var}^\ast(M_n^\ast) & = \frac{4\pi^2}{n} \sum_{j=1}^N\big( g_{\widehat{\theta}_0}(-\lambda_{j,n})   g^\top_{\widehat{\theta}_0}(-\lambda_{j,n})  + g_{\widehat{\theta}_0}(\lambda_{j,n})   g^\top_{\widehat{\theta}_0}(-\lambda_{j,n}) \\
& \ \ \ \  \ \ \ \ \ \ \ + g_{\widehat{\theta}_0}(-\lambda_{j,n})   g^\top_{\widehat{\theta}_0}(\lambda_{j,n})  +
g_{\widehat{\theta}_0}(\lambda_{j,n})   g^\top_{\widehat{\theta}_0}(\lambda_{j,n})\big)
\widehat{f}(\lambda_{j,n})^2\\
%& =  \frac{4\pi^2}{n} \sum_{j=1}^N \big(g_{\widehat{\theta}_0}(-\lambda_{j,n})   g^\top_{\widehat{\theta}_0}(-\lambda_{j,n}) + g_{\theta_0}(\lambda_{j,n})   g^\top_{\theta_0}(\lambda_{j,n})\big)\widehat{f}(\lambda_{j,n})^2 +o_P(1)\\
& =  \frac{8\pi^2}{n} \sum_{j\in{\mathcal G}(n)} g_{\theta_0}(\lambda_{j,n})   g^\top_{\theta_0}(\lambda_{j,n}) f(\lambda_{j,n})^2 +o_P(1)\\
& \stackrel{P}{\rightarrow} V_1,
\end{align*}
where the last equality  follows using   the symmetry of $  g_{\theta}(\lambda)$ with respect to $\lambda$,  assertion (\ref{eq.lem2-1}) and  Assumption 3. 

Consider (iii). We have 
\begin{align*}
\Sigma_n^+ & = \frac{b}{k}\sum_{\ell=1}^k{\rm Var}^*\Big( \frac{2\pi}{b}\sum_{j\in {\mathcal G}(b)} g_{\widehat{\theta}_0}(\lambda_{j,b}) I_b^{(\ell)}(\lambda_{j,b})\Big)\\
& = \frac{4\pi^2}{b}\sum_{j_1\in{\mathcal G}(b)} \sum_{j_2\in{\mathcal G}(b)}  g_{\widehat{\theta}_0}(\lambda_{j_1,b}) g^\top_{\widehat{\theta}_0}(\lambda_{j_2,b}) {\rm Cov}^\ast\big( I_b^{(i_1)}(\lambda_{j_1,b}),I_b^{(i_1)}(\lambda_{j_2,b}) \big)\\
& = \frac{4\pi^2}{b}\sum_{j_1\in{\mathcal G}(b)} \sum_{j_2\in{\mathcal G}(b)}  g_{\widehat{\theta}_0}(\lambda_{j_1,b}) g^\top_{\widehat{\theta}_0}(\lambda_{j_2,b}) 
\\
& \ \ \ \ \ \ \  \times \frac{1}{n-b+1}\sum_{t=1}^{n-b+1}\big\{I_b^{(t)}(\lambda_{j_1,b}),I_b^{(t)}(\lambda_{j_2,b} - \widetilde{f}(\lambda_{j_1,b}) \widetilde{f}(\lambda_{j_1,b} )  \big\}\\
&
= \frac{4\pi^2}{b}\sum_{j_1\in{\mathcal G}(b)} \sum_{j_2\in{\mathcal G}(b)}  g_{\widehat{\theta}_0}(\lambda_{j_1,b}) g^\top_{\widehat{\theta}_0}(\lambda_{j_2,b}) 
 {\rm Cov}\big( I_b^{(1)}(\lambda_{j_1,b}),I_b^{(1)}(\lambda_{j_2,b}) \big) + o_P(1)\\
\end{align*} 
where the $o_P(1)$ term follows using   Lemma 4.1  of Meyer et al. (2020) and Assumption 4. Thus using the covariance properties of the periodogram $ I_b^{(1)}(\lambda_{j,b})$ of the subsample $ X_1, X_2, \ldots, X_b$ and (\ref{eq.lem2-1}), we get  that, as  $b\rightarrow\infty$, 
\[ \frac{4\pi^2}{b}\sum_{j_1\in{\mathcal G}(b)} \sum_{j_2\in{\mathcal G}(b)}  g_{\widehat{\theta}_0}(\lambda_{j_1,b}) g^\top_{\widehat{\theta}_0}(\lambda_{j_2,b})  
 {\rm Cov}\big( I_b^{(1)}(\lambda_{j_1,b}),I_b^{(1)}(\lambda_{j_2,b}) \big) \stackrel{P}{\rightarrow} V_1 + V_2.\]
By the same lemma   we also have   
$$(n-b+1)^{-1}\sum_{t=1}^{n-b+1} I_b^{(t)}(\lambda_{j,b})^2/\widetilde{f}_b(\lambda_{j,b})^2 \stackrel{P}{\rightarrow} 2,$$
  which implies
using (\ref{eq.lem2-1}) again and Assumption 3, that  $ C_{n}^+ \stackrel{P}{\rightarrow} V_1$, in probability. The assertion follows then since    $ V^+_{2,n}= \Sigma^+_n-C_n^+$. 
\end{proof}

{\bf Proof of Theorem~\ref{th.bootWhittle}:} \    
In  view of the definition of $ L_n^\ast$ and  Lemma~\ref{le.Appendix2},  to establish  the assertion of the theorem it suffices to show that,  
\begin{equation}
\label{eq.L_nStar}
\sqrt{n}(\widehat{\theta}_n^\ast - \widehat{\theta}_0)  \stackrel{D}{\rightarrow} {\mathcal N}(0, W^{-1}V_1 W^{-1}).
\end{equation}
Since  $ D_n^{(2)}(\widehat{\theta}_0,I^\ast_n)= W^\ast_n$ we get  by Lemma~\ref{le.Appendix2}(i),  that  $ \| D_n^{(2)}(\widehat{\theta}_0, I^\ast_n) - W\|_F \stackrel{P}{\rightarrow} 0 $ and that,  for $n$ large enough,   $ D_n^{(2)}(\widehat{\theta}_0,I^\ast_n)$ is nonsingular; see also the proof of Lemma~\ref{le.Appendix1}(ii).
Recall  that  $ D^{(1)}_n(\widehat{\theta}_n^\ast,I^\ast_n) =0$ and  that $ D^{(1)}_n(\widehat{\theta}_0,\widehat{f}) =0 $. We 
 get using  the expansion 
(\ref{eq.ML-Taylor}), that 
\begin{align} \label{eq.proof-distr1}
\big[I_m +   \big(D_n^{(2)}(\widehat{\theta}_0,I^\ast_n)\big)^{-1} & R_n(\widehat{\theta}^\ast_n)\big]\sqrt{n}(\widehat{\theta}_n^\ast-\widehat{\theta}_0) \nonumber  \\
&  = - \big(D_n^{(2)}(\widehat{\theta}_0,I^\ast_n)\big)^{-1} \sqrt{n}\big( D^{(1)}_n(\widehat{\theta}_0,I^\ast_n)-  D^{(1)}_n(\widehat{\theta}_0,\widehat{f}) \big).
\end{align}
Since $ \big(D_n^{(2)}(\widehat{\theta}_0,I^\ast_n)\big)^{-1}\stackrel{P}{\rightarrow} W^{-1}$, in probability,   assertion (\ref{eq.L_nStar}) follows from 
expression  (\ref{eq.proof-distr1}),  if we show that 
\begin{equation}\label{eq.distr-part1}
-  \sqrt{n}\big( D^{(1)}_n(\widehat{\theta}_0,I^\ast_n)-  D^{(1)}_n(\widehat{\theta}_0,\widehat{f}) \big) \stackrel{D}{\rightarrow} {\mathcal N}(0,V_1),
\end{equation}
and
\begin{equation}\label{eq.distr-part2}
 R_n(\widehat{\theta}_n^\ast)=o_P(1),
\end{equation}
in probability.

For (\ref{eq.distr-part1}) we have, 
\begin{align*}
-  \sqrt{n}\big( D^{(1)}_n(\widehat{\theta}_0,I^\ast_n)-  D^{(1)}_n(\widehat{\theta}_0,\widehat{f}) \big)  & = -\frac{1}{\sqrt{n}}\sum_{j\in{\mathcal G}(n)} \frac{\partial}{\partial \theta} \frac{1}{f_\theta (\lambda_{j,n})}\Big|_{\theta=\widehat{\theta}_0}(I^\ast_n(\lambda_{j,n}) - \widehat{f}(\lambda_{j,n}))\\
& =\frac{1}{\sqrt{n}}\sum_{j=1}^N W_{j,n}  (U^\ast_j-1),
\end{align*}
where $ W_{j,n}=(g_{\widehat{\theta}_0}(\lambda_{j,n}) + g_{\widehat{\theta}_0}(-\lambda_{j,n}))\widehat{f}(\lambda_{j,n})$ and the $U^\ast_j$'s are i.i.d. (\ref{eq.distr-part1})  follows then because $ {\rm E}^\ast(W_{j,n}  (U^\ast_j-1))=0$,  $ {\rm Var}^\ast(W_{j,n}  (U^\ast_j-1))\stackrel{P}{\rightarrow} V_1$ by the same arguments as those used in the proof that $ {\rm Var}^\ast(M_n^\ast) \stackrel{P}{\rightarrow} V_1$ in Lemma~\ref{le.Appendix2}(ii),
and because 
\begin{align*}
\sum_{j=1}^N {\rm E}^\ast\|\frac{1}{\sqrt{n}}W_{j,n}  (U^\ast_j-1) \|^3 &  ={\rm E}^\ast(|U^\ast_j-1|^3)\frac{1}{n^{3/2}}\sum_{j=1}^N \|W_{j,n}\|^3=O_P(n^{-1/2}),
\end{align*}
verifies Liapunov's condition. 

Consider (\ref{eq.distr-part2}) and observe that $  R_n(\widehat{\theta}^\ast_n)=D^{(2)}_n(\theta^+_n,I^\ast_n)-D^{(2)}_n(\widehat{\theta}_0,I^\ast_n)$ for  some $\theta^+_n\in \Theta$ such that $ \|\theta^+_n-\widehat{\theta}_0\|\leq \|\widehat{\theta}^\ast_n-\widehat{\theta}_0\|$. Denote by $ d_{j,k}(\mathring \theta)$ the $(j,k)$th element of the matrix
$ D^{(2)}_n(\mathring \theta,I^\ast_n)$, that is, $  d_{j,k}(\mathring \theta) = \partial^2\big/(\partial \theta_j\partial \theta_k) D_n(\theta,I^\ast)|_{\theta=\mathring\theta}$. We then have for the $(j,k)$th element  $ r_{j,k}(\widehat{\theta}_n^\ast)$ of the matrix $ R_n(\widehat{\theta}^\ast_n)$,
\begin{align*}
 | r_{j,k}(\widehat{\theta}_n^\ast)| & =  |d_{j,k}(\theta^+_n) -d_{j,k}(\widehat\theta_0)|  \\
 % & \leq  \frac{1}{n}\sum_{j\in{\mathcal G}(n)} \big|  \frac{\partial^2}{\partial\theta\partial\theta^\top} \log f_\theta(\lambda_{j,n})\Big|_{\theta=\theta^+_n} -  \frac{\partial^2} {\partial\theta\partial\theta^\top} \log f_\theta(\lambda_{j,n})\Big|_{\theta=\widehat\theta_0}\big|\\
%& \  \ + \frac{1}{n}\sum_{j\in{\mathcal G}(n)} \big| \frac{\partial^2}{\partial\theta\partial\theta^\top} f^{-1}_\theta(\lambda_{j,n})\Big|_{\theta=\theta^+_0}
%- \frac{\partial^2}{\partial\theta\partial\theta^\top} f^{-1}_\theta(\lambda_{j,n})\Big|_{\theta=\widehat\theta_0}\big| I_n^\ast(\lambda_{j,n})\\
& \leq \sup_{\lambda\in[-\pi,\pi]} \big|  \frac{\partial^2}{\partial\theta_j\partial\theta_k} \log f_\theta(\lambda)\Big|_{\theta=\theta^+_n} -  \frac{\partial^2}{\partial\theta_j\partial\theta_k} \log f_\theta(\lambda)\Big|_{\theta=\widehat\theta_0}\big|\\
& \ \ + \sup_{\lambda\in[-\pi,\pi]} \big| \frac{\partial^2}{\partial\theta_j\partial\theta_k} f^{-1}_\theta(\lambda)\Big|_{\theta=\theta^+_0}
- \frac{\partial^2}{\partial\theta_j\partial\theta_k} f^{-1}_\theta(\lambda)\Big|_{\theta=\widehat\theta_0}\big|\frac{1}{n}\sum_{j\in{\mathcal G}(n)}I_n^\ast(\lambda_{j,n}) \\
&=o_{P}(1),
\end{align*}
as $ n \rightarrow\infty$, since $ n^{-1}\sum_{j\in{\mathcal G}(n)}I_n^\ast(\lambda_{j,n})=O_P(1)$, the second order partial derivative functions  $\partial^2\big/(\partial\theta_j\partial\theta_k) \log f_\theta(\lambda)$
and $ \partial^2\big/(\partial\theta_j\partial\theta_k) f^{-1}_\theta(\lambda)$  are uniformly continuous of the compact set $ [-\pi,\pi]\times\Theta$ and
$ \|\theta^+_n-\widehat{\theta}_0\| \leq  \|\widehat{\theta}^\ast_n-\widehat{\theta}_0\|  \stackrel{P}{\rightarrow} 0$, 
in probability, by Lemma~\ref{le.Appendix1}(ii). \hfill$\Box$

\renewcommand{\baselinestretch}{1.1}

\end{document}